\documentclass[a4paper,11pt] {article}
\usepackage{amsmath,amsthm}
\usepackage[english]{inputenc}
\usepackage{geometry}
\usepackage{mdwlist}
\numberwithin{equation}{theo}

\numberwithin{equation}{propo}
\usepackage{amssymb}
\usepackage{tikz}
\usepackage{titlesec}
\usetikzlibrary{decorations}
\usetikzlibrary{decorations.markings,arrows}
\usepackage{cite}
\usepackage{comment}
\usepackage{mathtools}
\usepackage{overpic}
\usepackage{amsmath,amsthm}%
\usepackage{etex}
\usepackage{txfonts}
\usepackage{amsfonts}%
\usepackage{graphicx}
\usepackage{subfig}

\newtheorem{ass}{Assumption}
\newtheorem*{ass*}{Assumption}
\newtheorem{rmq}{Remark}
\newtheorem{df}{Definition}
\newtheorem{thm}{Theorem}
\newtheorem*{thm1*}{thm 1}
\newtheorem*{thm2*}{thm 2}
\newtheorem*{thm3*}{thm 3}
\newtheorem{prop}{Proposition}
\newtheorem*{thm*}{Theorem}
\newtheorem*{prop*}{Proposition}
\newtheorem*{prop2*}{prop 2}
\newtheorem*{prop3*}{prop 3}
\newtheorem*{prop5*}{prop 5}
\newtheorem{lemma}{Lemma}

\newtheorem{cor}{Corollary}
\titleformat\section{}{}{0pt}{\Large\scshape\bfseries\filcenter\thesection{} - }

\numberwithin{equation}{section}

\begin{document}

\title{On the non-transverse homoclinic channel of a center manifold}
\author{Cezary Olszowiec, Dmitry Turaev} 
\date{}

\maketitle



\noindent
\let\thefootnote\relax\footnote{{
\hspace{-0.6cm} cezary.olszowiec@gmail.com \\
}}


\begin{abstract}
We consider a scenario when a stable and unstable manifolds of compact center manifold of a saddle-center coincide. The normal form of the ODE governing the system near the center manifold is derived and so is the normal form of the return map to the neighbourhood of the center manifold.
The limit dynamics of the return map is investigated by showing that it might take the form of a Henon-like map possessing a Lorenz-like attractor or satisfy 'cone-field condition' resulting in partial hyperbolicity. We consider also motivating example from game theory.
\end{abstract}

\section{Introduction}
\subsection{Outline}
In this paper we investigate a toy-model consisting of a saddle-center in $\mathbb R^4$ and its $2$ - dimensional compact center manifold which $3$ - dimensional stable and unstable manifolds coincide, i.e. they create a non-transverse homoclinic channel, later denoted by $I$ (see fig. \ref{fig:toyModel}). We derive normal form of the return map to the vicinity of the center manifold along $I$ and under some additional assumptions investigate its dynamics. Although all of the orbits starting in the vicinity of $I$ accumulate on $I$, the limit dynamics of the return map might be very complicated. Even in the very simple scenario the return map might turn out to be a $3$ - dimensional Henon-like map possessing a Lorenz-like attractor (see \cite{LorenzLike1}, \cite{LorenzLike2}). The return map can also satisfy the so-called 'cone-field condition' and hence be partially hyperbolic in the originally investigated coordinates, which leads to the existence of smooth invariant foliation. Factorizing along the leaves of this foliation reduces the problem to the investigation $2$-dimensional mapping.
\newline
The motivation for studying this toy-model comes from the bimatrix Rock-Scissors-Paper game. In \cite{Olszowiec1} the problem of describing the dynamics on the boundary of the codimension $1$ was considered. In RSP model the existence of $2$ - dimensional invariant manifold with connecting non-transverse homoclinic channel naturally appears for all considered parameter values. In fact this $2$-dimensional invariant manifold has $6$ - connected components and they are connected with the corresponding $6$ pieces of the non-transverse homoclinic channel. Each of these pieces lies in a different invariant codimension $1$ subspace. In this paper we provide general picture of RSP model as an example in the section \ref{RSPsub} as well as relation to the toy-model we consider. In particular we touch the questions arising in RSP model including the form of the scattering maps defined by the pieces of the non-transverse homoclinic channel and dynamics in the vicinity of this channel, i.e. the possibility of reduction the problem to the study of the scattering maps and infinite shadowing of their composition.
\newline
Despite that in the RSP model the homoclinic channel is non-transverse and the general theory of scattering maps assumes transversality condition to hold (see \cite{Scatt}), the considered homoclinic channel in fact fulfills the invariant codimension $1$ invariant subspaces, hence it makes it possible to consider scattering maps in the RSP model. However in the absence of the transversality condition it is not possible to utilise topological tools \cite{Cov1} to ensure infinite shadowing of the scattering maps as in \cite{Scatt}, \cite{ShadowingLemma}, \cite{Diffusion1}. Furthermore, in comparison to our toy-model we provide explanation why one should not expect infinite shadowing of the scattering maps in the RSP model.
\newline
Scattering maps are a powerful geometric tool for investigations especially in diffusion problems in Hamiltonian systems (scattering map becomes symplectic in this case). However, they are not defined in the case of intersection of invariant manifolds being robustly non-transvere. The problem of shadowing along such a non-transverse connections with the usage of topological tools was investigated in \cite{Adria}, which is also related to the problems from RSP model \cite{Olszowiec1}.
\newline
Although in theory one gets rid off the issues with the non-transversality via perturbation methods, in practice the problems of non-transversality have to be dealt with and appear in real systems, e.g. in Celestial Mechanics.
As an example we refer to \cite{Masd}, where it has been found, that the full set of heteroclinic connections between invariant tori in a fixed energy level of the Hill's model is bounded by the heteroclinic connection between periodic orbits and also by a pair of non-transverse heteroclinic connections of the tori (in fact the invariant manifolds of the tori are tangent along this connection)$.^*$\footnote{*We thank Josep-Maria Mondelo for this reference.} Our paper also adresses this type of questions. In fact we show that even in the simplest toy-model with additional restrictive assumptions, one should expect very complicated behaviour, e.g. existence of the Lorenz-like attractor for the return map which in suitable coordinates can be presented as a $3$-dimensional Henon-like map. It was proven in \cite{LorenzLike1}, \cite{LorenzLike2} that under some additional assumptions such a Henon-like map can be approximated by the time - $1$ map of the flow governed by the Shimizu-Morioka system. The existence of the geometric Lorenz attractor for the Shimizu-Morioka system was recently proven in \cite{Shimizu}.

\subsection{The organisation of the paper}
In the remaining part of this section we briefly introduce considered model and state the main results of the paper.
\newline
In the section \ref{examples} we describe an example related to the considered toy-model, i.e. Rock-Scissors-Paper game.
\newline
In the section \ref{nf1} we derive the normal form of the ODE, governing considered toy-model, near the center manifold.
\newline
Section \ref{nf2} contains derivation of the normal form of the return map to the cross section in the vicinity of the center manifold.
\newline
In \ref{foliation}, under some additional assumption we prove the existence of the foliation theorem resulting in the description of the limit dynamics of the return map. In this case we prove existence of a Lorenz-like attractor for the return map.
\newline
We show in section \ref{partialHyp} that the return map might satisfy a 'cone-field condition', resulting in its partial hyperbolicity and existence of the foliation, which allows us to factorize the $3$-dimensional return map along the leaves of this foliation. 

\subsection{Setting of the problem and statement of the results}
We consider a saddle-center in $\mathbb R^4$ and its $2$ - dimensional center manifold $D$ foliated by periodic orbits. We assume that the center manifold possesses the $3$ - dimensional stable and unstable manifolds which coincide (see fig. \ref{fig:toyModel}). We prove the following theorem about the normal form of this system near the center manifold:
\begin{thm*}
There exists local coordinates $(v_1, v_2, r, \phi) \in \mathbb R^4$ near the center manifold $D$ in which $(0,0,0,0)$ is the saddle-center and the set $\{(0,0, r, \phi) \ | \ r \in [0,1], \phi \in [0,1) \}$ is the center manifold $D$. The $(r,\phi)$ are polar coordinates. The coordinates $v_1,v_2$ correspond to the locally unstable and stable directions. The normal form of the ODE describing this model, near the center manifold is:
\begin{equation*}
\begin{cases}
\dot{v_1}= p(r) \cdot v_1  + p_0(v_1,v_2,r,\phi) \cdot {v_1}^2 v_2
\\
\dot{v_2}= \varsigma(r) \cdot v_2 +  \varsigma_0(v_1,v_2,r,\phi) \cdot v_1 {v_2}^2
\\
\dot{r} = v_1 v_2 \cdot r_{0}(v_1,v_2,r,\phi) 
\\
\dot{\phi} = \omega(r) + v_1 v_2 \cdot \omega_{0}(v_1,v_2,r,\phi) 
\end{cases}
\end{equation*}
with $p_0$, $\varsigma_0$, $r_0$, $\omega_0$ being $1$ - periodic in $\phi$.
\end{thm*}
We derive the normal form of the return map along the non-transverse homoclinic channel $I$, created formed by the coincidence of stable and unstable manifolds of $D$:
\begin{thm*}
The normal form of the return map to the cross section $S_1 := \{ v_1 = h \}$ is
\begin{equation*}
\left(
\begin{array}{c}
v_2\\
r\\
\phi
\end{array}
\right)
\mapsto
\left(
\begin{array}{c}
h \cdot  \Big(\frac{v_2 \cdot a_1(r,\phi)}{h}  \Big)^{\cdot \frac{- \varsigma(b_0(r,\phi))}{p(b_0(r,\phi))}} + O(v_2 \cdot ln(v_2)) \\
b_0(r,\phi) + O(v_2) \\
c_0(r,\phi)  - \frac{\omega(b_0(r,\phi))}{p(b_0(r,\phi))} \cdot ln(\frac{v_2}{h}) + O(v_2 \cdot ln(v_2)) \ \ (mod \ 1)
\end{array}
\right)
\end{equation*}
with $b_0$, $c_0$, $a_1$ being $1$ - periodic in $\phi$.
\end{thm*}
We investigate the dynamics of the return map under condition $\varsigma(r)+p(r)<0$:
\begin{thm*}
If $\varsigma(r)+p(r)<0$, then there exists a neighbourhood of the homoclinic channel $I$, such that orbit of each point from this neighbourhood accumulates on $I$.
\end{thm*}
Furthermore we investigate the dynamics of the return map after the coordinate transformation $z \approx ln(v_2)$ 
\begin{thm*}
The normal form of the return map along $I$, in the coordinates $(z,r,\phi)$ is:
\begin{equation}\label{nontruncIntro}
\left(
\begin{array}{c}
z\\
r\\
\phi
\end{array}
\right)
\mapsto
\left(
\begin{array}{c}
\Omega(r,\phi) + \Gamma(r,\phi) \cdot z + O(z \cdot e^z)\\
b_0(r, \phi) + O(e^z)\\
c(r,\phi) + z + O(z \cdot e^z)  \ \ (mod \ 1)
\end{array}
\right)
\end{equation}
with $\Omega$, $\Gamma(r,\phi)=-\frac{\varsigma(b_0(r,\phi))}{p(b_0(r,\phi))}$, $b_0$, $c$ being $1$ - periodic in $\phi$.
\end{thm*}

Assuming that $-\frac{\varsigma(b_0(r,\phi))}{p(b_0(r,\phi))}=\Gamma(r,\phi)=\Gamma \in \mathbb N_2$, we prove that in the limit with $v_2 \rightarrow 0$, the dynamics of the return map are the same as the dynamics of its truncated version:
\begin{equation}\label{truncIntro}
\left(
\begin{array}{c}
z\\
r\\
\phi
\end{array}
\right)
\mapsto
\left(
\begin{array}{c}
\Omega(r,\phi) + \Gamma \cdot z \ (mod \ 1)\\
b_0(r, \phi) \\
c(r,\phi) + z \ \ (mod \ 1)
\end{array}
\right)
\end{equation}

This is due to the following theorem we prove:

\begin{thm*}
For the mapping $\varpi$:
\begin{equation}\label{varpi}
\begin{aligned}
\varpi: \ \mathbb R \times \mathbb R \times \mathbb R \times S^1 \ni \left( \begin{array}{c}
v_2 \\
z \\
r \\ 
\phi \\ 
\end{array} \right) 
\mapsto
\left( \begin{array}{c}
h \cdot \big(\frac{a_1(r,\phi) \cdot v_2}{h} \big)^{\frac{-\varsigma(b_0(r,\phi))}{p(b_0(r,\phi))}} + h.o.t.(v_2) \\
\Omega(r,\phi) + \Gamma(r,\phi) \cdot z + O( v_2 \cdot \ln(v_2)) \ (mod \ 1)\\
b_0(r,\phi) + O( v_2  ) \\
c(r,\phi) + z + O( v_2 \cdot \ln(v_2) ) \ \ (mod \ 1 \ )  \\ 
\end{array} \right) \in \mathbb R \times \mathbb R \times \mathbb R \times S^1 
\end{aligned}
\end{equation}
there exists a $C^1$ and $\varpi$ - invariant foliation $\mathcal{L} = \{h(z,r,\phi)\}$ of the space $(v_2,z,r,\phi)$, with $C^2$ leaves given by the graphs of the functions $v_2 = h(z,r,\phi)$.
\end{thm*}
So there is a one-to-one correspondence between the orbits under the return map (\ref{nontruncIntro}) and truncated return map (\ref{truncIntro}).

It turns out that the coefficients $\Omega(r,\phi), \Gamma, b_0(r, \phi), c(r,\phi)$ can be choosen in such a way that the truncated return map is conjugated to the $3$ - dimensional Henon-like map possessing Lorenz-like attractor.

\begin{prop*}
There exist coefficients $\Omega(r,\phi)$, $\Gamma$, $b(r,\phi)$, $c(r,\phi)$, such that the mapping
\begin{equation}\label{}
\begin{aligned}
T: \ \ \left( \begin{array}{c}
z \\
r \\ 
\phi \\ 
\end{array} \right) 
\mapsto
\left( \begin{array}{c}
\Omega(r,\phi) + \Gamma \cdot z \ \ (mod \ 1) \\
b(r,\phi) \\
c(r,\phi) + z \ \ (mod \ 1) \\ 
\end{array} \right) 
\end{aligned}
\end{equation}
has a fixed point with a normal form, in a suitable coordinates, given by:
\begin{equation}\label{}
\begin{aligned}
\left( \begin{array}{c}
x \\
y \\ 
w \\ 
\end{array} \right) 
\mapsto
\left( \begin{array}{c}
y \\
w \\
x + y - w + A \cdot y^2 + B \cdot y \cdot w + C \cdot w^2 + h.o.t.
\end{array} \right) 
\end{aligned}
\end{equation}
with $(C - A) \cdot (A - B + C) > 0$.
Hence, by \cite{LorenzLike1}, \cite{LorenzLike2} and \cite{Shimizu} there exists arbitrarily small perturbation of this mapping possessing a Lorenz-like attractor.
\end{prop*}

Finally we investigate partial hyperbolicity of the truncated return map:
\begin{prop*}\label{coneField}
There exist coefficients $\Omega(r,\phi)$, $\Gamma(r,\phi)$, $b(r,\phi)$, $c(r,\phi)$, such that the mapping $T$ possesses an invariant cone-field 
$$\mathcal{C}_{L, (z,r,\phi)}:= \Big{\{} v= v_1 + v_{23} \in T_{(z,r,\phi)}M \ | \ v_1=(v_{11},0,0), \ v_{23}=(0,v_2,v_3) \ \text{such that } ||v_{23}|| < L\cdot |v_1| \Big{\}}$$ 
where $M := \mathbb R \times [0,1] \times S^1$, $L<1$ and $|| \cdot ||$ denotes the Euclidean norm. 
\end{prop*}
In consequence, truncated return map is partially hyperbolic and so there exists an invariant foliation with leaves of the form $(r,\phi)=h(z)$. So one can factorize the truncated 
return map along the leaves in order to reduce its study to the $2$-dimensional mapping.

\section{Non-transverse homoclinic channel in the Rock-Scissors-Paper game}\label{examples}
\subsection{Basic facts about scattering maps for autonomous flows}
In this section, as a background to the following ones, we briefly introduce (after \cite{Scatt}) the notion of the scattering map for autonomous flows.
\newline
For our purposes, let us assume that $\Phi$ is a $C^{\infty}$ smooth flow on a manifold $M$ and $\Lambda \subset M$ let be a normally hyperbolic invariant manifold. 
\newline
Firstly let us recall the definitions of the stable and unstable manifolds of $\Lambda$ and of $x \in \Lambda$
$$W^s_{\Lambda} := \Big{\{} y \in M \ | \ d(\Phi_t(y), \Lambda) \leq C_y e^{-\beta \cdot t}, \ t \geq 0  \Big{\}}, \ \ \ \ W^u_{\Lambda} := \Big{\{} y \in M \ | \ d(\Phi_t(y), \Lambda) \leq C_y e^{-\beta \cdot |t|}, \ t \leq 0 \Big{\}} $$
$$W^s_{x} := \Big{\{} y \in M \ | \ d(\Phi_t(y), \Phi_t(x)) \leq C_{x,y} e^{-\beta \cdot t}, \ t \geq 0  \Big{\}}, \ \ \ \ W^u_{x} := \Big{\{} y \in M \ | \ d(\Phi_t(y), \Phi_t(x)) \leq C_{x,y} e^{-\beta \cdot |t|}, \ t \leq 0  \Big{\}} $$
The constant $\beta >0$ is an expansion rate from $\Lambda$.
\newline
Assume that there exists a manifold $\Gamma \subset W^s_{\Lambda} \cap W^u_{\Lambda}$, along which $W^s_{\Lambda}$, $W^u_{\Lambda}$ intersect transversally, i.e. for every $x \in \Gamma$:
$$T_x M = T_x W^s_{\Lambda} + T_x W^u_{\Lambda}, \ \ \ \ \ \ T_x W^s_{\Lambda} \cap T_x W^u_{\Lambda} = T_x \Gamma$$
For the definition of the scattering map, it is also required to assume that $\Gamma$ is transverse to the $W^s_x$, $W^u_x$ foliations, i.e. for every $x \in \Gamma$ we have
$$T_x \Gamma \oplus T_x W^s_{x_+} = T_x W^s_{\Lambda} \ , \ \ \ \ \ \ T_x \Gamma \oplus T_x W^u_{x_-} = T_x W^u_{\Lambda}$$
The wave operator is defined as the mapping:
$$\Omega_{\pm}: W_{\Lambda}^{s,u} \ni x \mapsto x_{\pm} \in \Lambda$$
where $x_{\pm}$ is such that 
$$| \Phi_t(x)- \Phi_t(x_{\pm})| \leq C_{x,x_{\pm}} e^{-\beta \cdot |t|} \text{, \  \ as \ \ \ } t \rightarrow \pm \infty$$ 
We say that $\Gamma$ is a homoclinic channel if intersection of $W^s_{\Lambda}$, $W^u_{\Lambda}$ is transverse along $\Gamma$ and that $\Gamma$ is transverse to the $W^s_x$, $W^u_x$ foliations. Moreover we require from the wave operators 
$$  \Omega_{\pm}^{\Gamma} := \big( {\Omega_{\pm} \big)\big{|}}_{\Gamma} : \Gamma \rightarrow \Omega_{\pm}^{\Gamma} \Big( \Gamma \Big) \subset \Lambda$$ 
to be $C^{l-1}$ diffeomorphisms. Here $l$ is the smoothness of the normally hyperbolic invariant manifold $\Lambda$.
\begin{df}
Assume that $\Gamma$ is a homoclinic channel. The scattering map is defined as
$$\sigma^{\Gamma} := \Omega_{+}^{\Gamma} \circ \Big( \Omega_{-}^{\Gamma} \Big)^{-1}$$
\end{df}

\subsection{Rock-Scissors-Paper game}
Let us recall the basic facts about Rock-Scissors-Paper game. We consider bimatrix game with payoff matrices $(A,B^T)$, where:
\begin{equation*}
\begin{aligned}
A= \left( \begin{array}{ccc}
\epsilon_X & 1 & -1 \\
-1 & \epsilon_X & 1 \\
1 & -1 & \epsilon_X \end{array} \right) 
\ \ \ \
B= \left( \begin{array}{ccc}
\epsilon_Y & 1 & -1 \\
-1 & \epsilon_Y & 1 \\
1 & -1 & \epsilon_Y \end{array} \right) 
\end{aligned}
\end{equation*}
$\epsilon_X,\epsilon_Y \in (-1,1)$ are the rewards for ties.
\newline
Assuming perfect memory of the players $X$, $Y$, the dynamics are governed by the coupled replicator equations (see \cite{Sato3}): 
\begin{equation*}
\begin{cases}
\dot{x_i}=x_i [(Ay)_i-x^T Ay]
\\
\dot{y_j}=y_j[(Bx)_j-y^T Bx]
\end{cases}
\end{equation*}
with $i,j=1,2,3$, and $x_1 + x_2 + x_3 =1$, $y_1 + y_2 + y_3 =1$, where $x_i, y_i$ denote the probabilities of playing strategy $i$ by players $X$ and $Y$, respectively. We investigate these equations as a system of ODE's in $\mathbb R^4$:
\begin{equation*}
\begin{cases}
\dot{x_1}=x_1 [(Ay)_1-x^T Ay] 
\\ 
\dot{x_2}=x_2 [(Ay)_2-x^T Ay]
\\ 
\dot{y_1}=y_1[(Bx)_1-y^T Bx] 
\\
\dot{y_2}=y_2[(Bx)_2-y^T Bx]
\end{cases}
\end{equation*}
with substitutions $x_3=1-x_1-x_2, y_3=1-y_1-y_2$ and constraints $x_1,x_2 \geq 0, x_1+x_2 \leq 1$, $y_1,y_2 \geq 0, y_1+y_2 \leq 1$.
\begin{rmq}
If each of the players is restricted to only two strategies, then either all orbits are periodic and spiral around one of the $6$ equilibria or they converge to one of the draw equilibrium $(R,R),(P,P),(S,S)$ with time tending to $\pm \infty$. The latter case occurs for the system restricted to the subspaces $\{x_i=0, \ y_i=0 \}$, $i=1,2,3$. 
\end{rmq}
The following was already investigated in \cite{Olszowiec1}:
\begin{prop}\label{6addEq}
There are 6 equilibria $Z^a,Z^b,Z^c,Z^d,Z^e,Z^f$ of (1,2,1) type (1-dim. stable manifold, 2-dim. center manifold, 1-dim. unstable manifold), which are centers within the 2 dimensional invariant subspaces $H^a,H^b,H^c,H^d,H^e,H^f$ (where the system is integrable), and have one dimensional stable and unstable manifolds $W^s_a,W^u_a,W^s_b,W^u_b$, $W^s_c,W^u_c,W^s_d,W^u_d,W^s_e,W^u_e,W^s_f,W^u_f$: 
\newline
$$Z^a= \Big(0, \frac{2}{3-\epsilon_Y},\frac{1+\epsilon_X}{3+\epsilon_X},\frac{2}{3+\epsilon_X}\Big), H^a= \{x_1=0,y_3=0 \}, W^s_a \subset \{ x_1=0 \}, W^u_a \subset \{ y_3=0 \}$$

$$Z^b= \Big( \ 0, \ \frac{1+\epsilon_Y}{3+\epsilon_Y}, \ \frac{1-\epsilon_X}{3-\epsilon_X}, \ 0 \ \Big), \  H^b =\{x_1=0,y_2=0\}, \  W^s_b \subset \{ y_2=0 \}, \ W^u_b \subset \{ x_1=0 \}$$  

$$Z^c= \Big( \ \frac{1-\epsilon_Y}{3-\epsilon_Y}, \ 0 , \ 0, \ \frac{1+\epsilon_X}{3+\epsilon_X} \Big), \  H^c= \{x_2=0,y_1=0 \}, \ W^s_c \subset \{ x_2=0 \}, \ W^u_c \subset \{ y_1=0 \}$$

$$Z^d= \Big(\frac{2}{3+\epsilon_Y}, 0,\frac{2}{3-\epsilon_X},\frac{1-\epsilon_X}{3-\epsilon_X}\Big), H^d = \{ x_2=0,y_3=0 \}, W^s_d \subset \{ y_3=0 \}, W^u_d \subset \{ x_2=0 \}$$

$$Z^e= \Big(\frac{1+\epsilon_Y}{3+\epsilon_Y}, \frac{2}{3+\epsilon_Y},0,\frac{2}{3-\epsilon_X}\Big), H^e = \{x_3=0,y_1=0\}, W^s_e \subset \{ y_1=0 \}, W^u_e \subset \{ x_3=0 \}$$

$$Z^f= \Big( \frac{2}{3-\epsilon_Y}, \frac{1-\epsilon_Y}{3-\epsilon_Y},\frac{2}{3+\epsilon_X},0 \Big), H^f = \{x_3=0,y_2=0 \}, W^s_f \subset \{ x_3=0 \}, W^u_f \subset \{ y_2=0 \}$$
\end{prop}
\begin{rmq}
On the boundary of codimension 1, the orbits travel from one invariant hyperplane $H^{\sigma}$ 
\newline
(with $\sigma \in \{a, \ b, \ c, \ d, \ e, \ f \}$) of dimension 2 to another one.
\end{rmq}
The following observation is based on the numerical simulations performed within the invariant subspaces $\{x_i=0\}$, $\{ y_j=0\}$ and as well in the $4$ dimensional neighbourhood of the subspaces $H^a$,...,$H^f$.
\begin{rmq}\label{non-transverse}
For all $\epsilon_X, \epsilon_Y \in (-1,1)$, the following inclusions hold:
$$W^s(H^a) = W^u(H^b) \subset \{x_1 = 0 \}, \ W^s(H^b) = W^u(H^f) \subset \{y_2 = 0 \}, \ W^s(H^f) = W^u(H^e) \subset \{x_3 = 0 \}$$
$$W^s(H^e) = W^u(H^c) \subset \{y_1 = 0 \}, \ W^s(H^c) = W^u(H^d) \subset \{x_2 = 0 \}, \ W^s(H^d) = W^u(H^a) \subset \{y_3 = 0 \}$$
\end{rmq}


\begin{figure}[!htb]
\minipage{0.23\textwidth}
  \includegraphics[width=\linewidth]{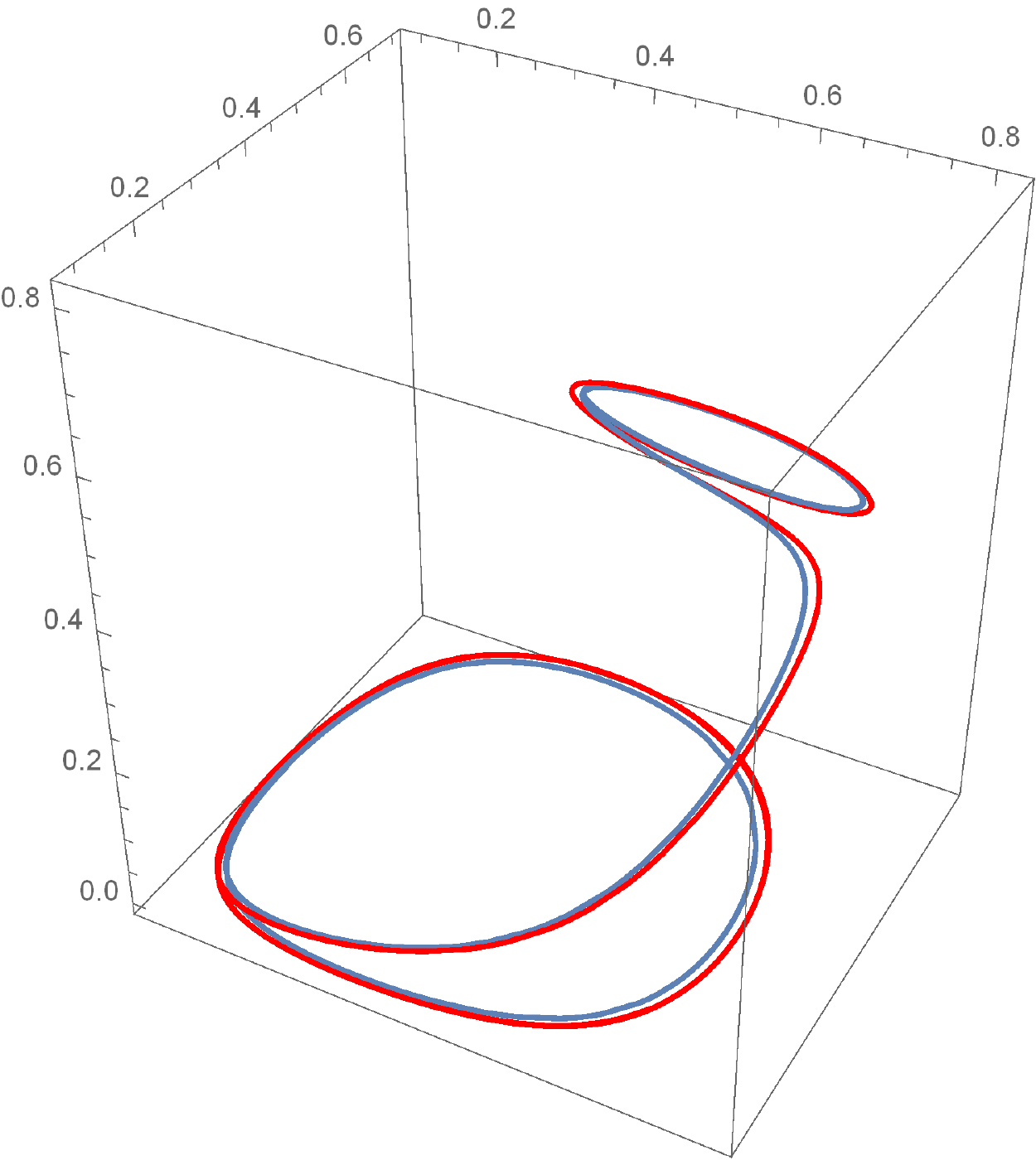}
	{\small (a) $\epsilon_X= -0.09, \ \epsilon_Y= -0.79$}
\endminipage\hfill
\minipage{0.23\textwidth}
  \includegraphics[width=\linewidth]{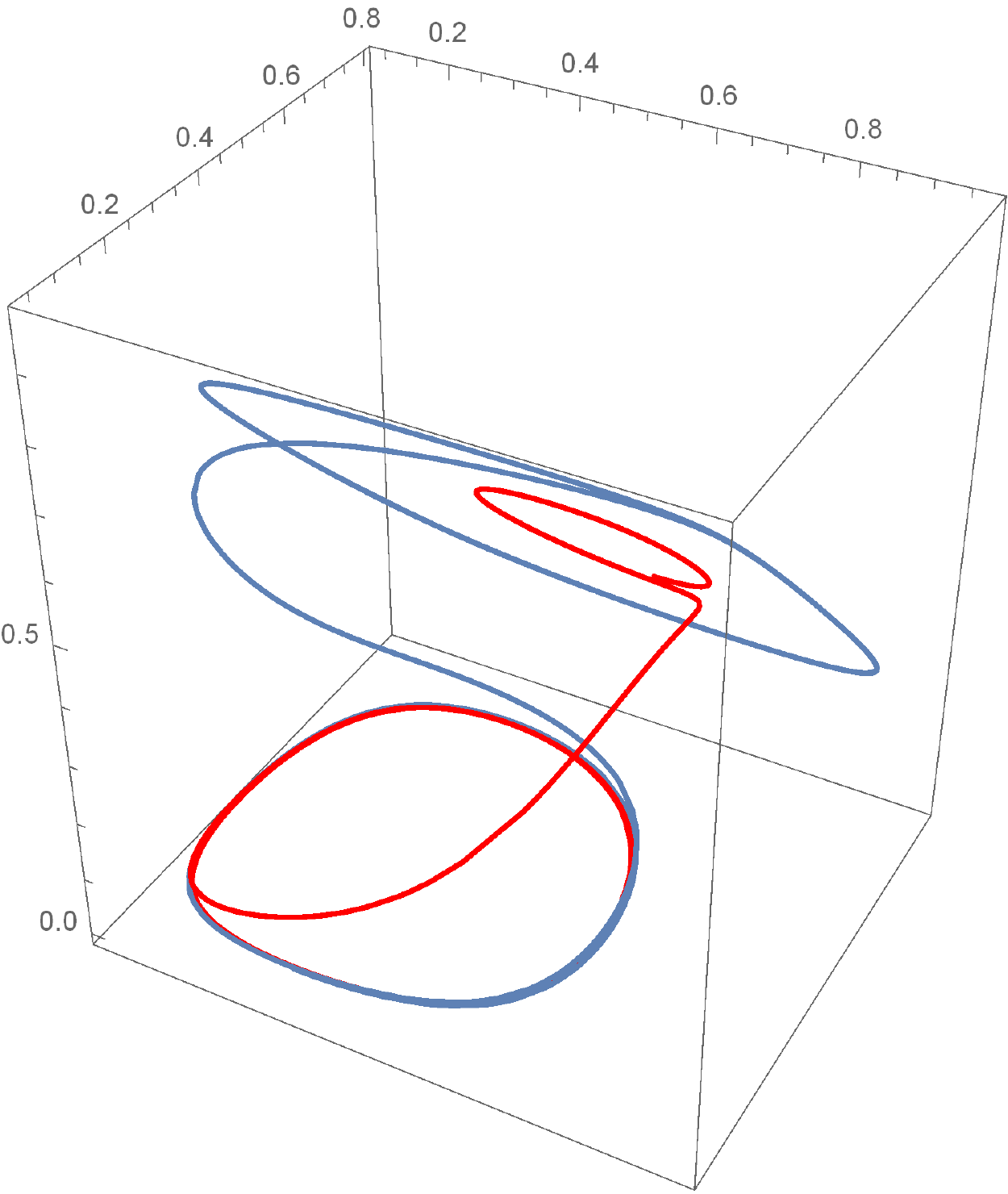}
	{\small (b) $\epsilon_X= 0.81, \ \epsilon_Y= 0.11$}
\endminipage\hfill
\minipage{0.23\textwidth}%
  \includegraphics[width=\linewidth]{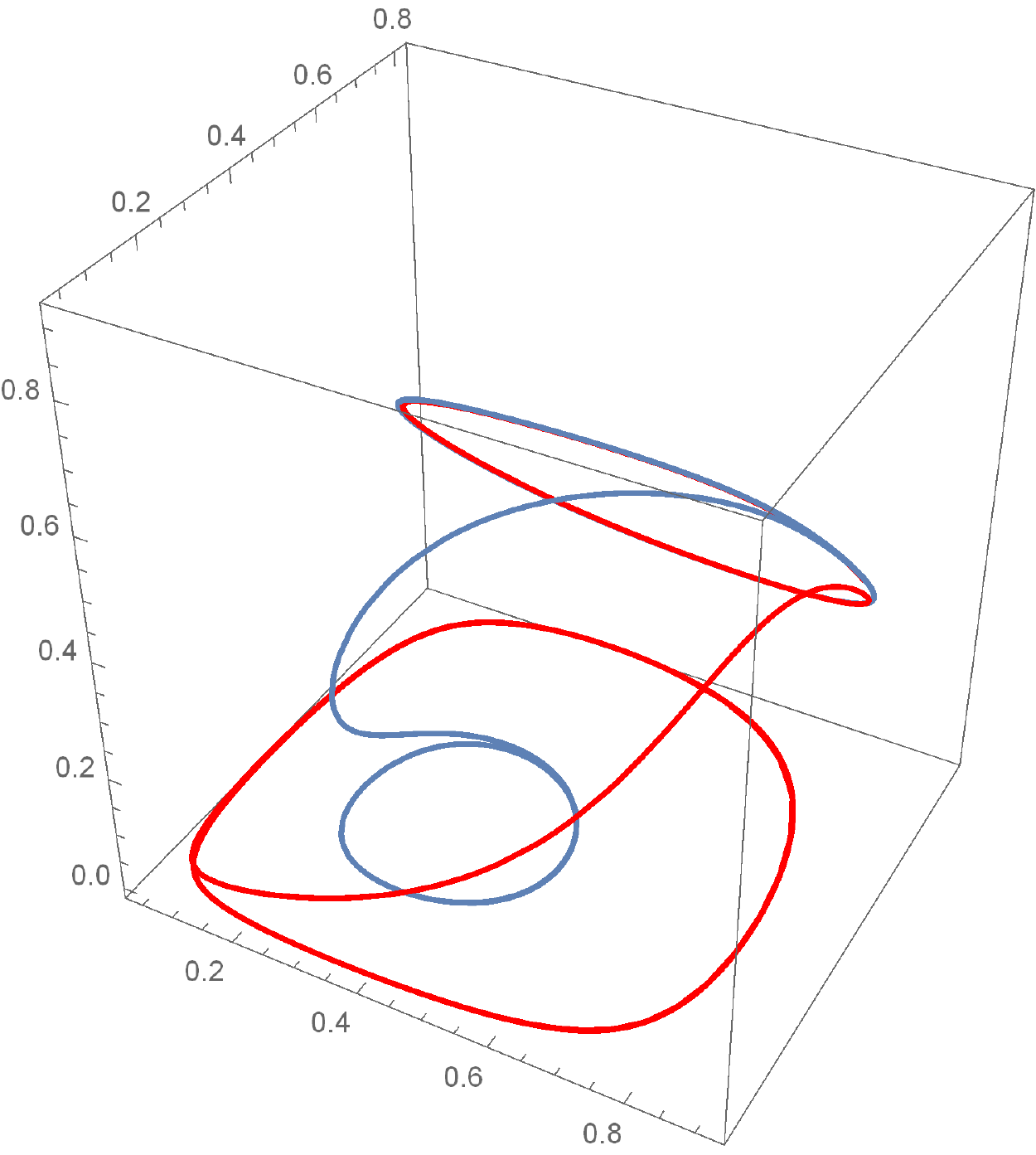}
	{\small (c) $\epsilon_X= 0.41, \ \epsilon_Y= 0.81$}
\endminipage\hfill
\minipage{0.23\textwidth}%
  \includegraphics[width=\linewidth]{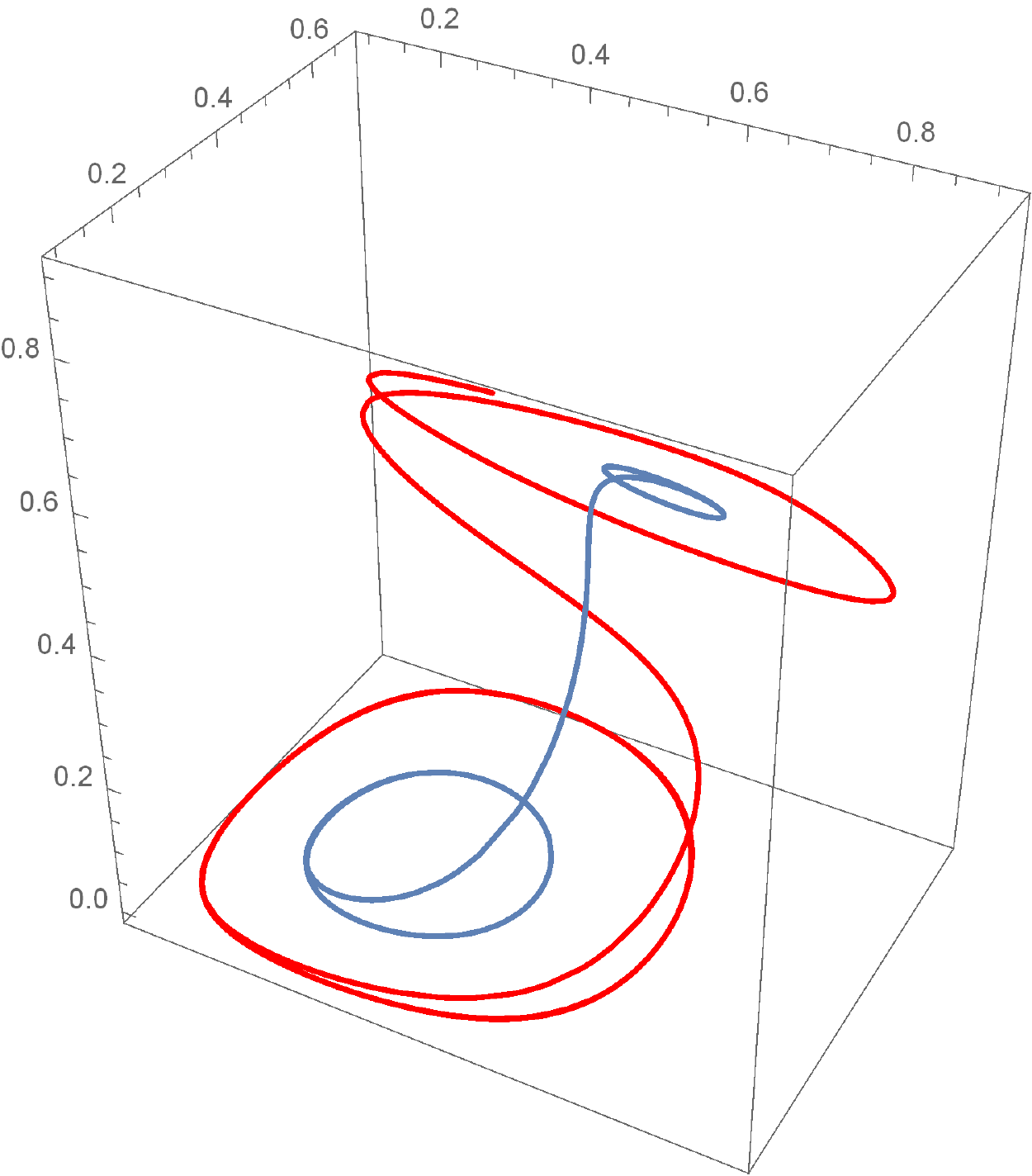}
	{\small (d) $\epsilon_X= -0.79, \ \epsilon_Y= -0.29$}
\endminipage
\caption{Generic types of behaviour of the flow on the boundary $\{ x_1=0 \}$, red curve is a a backward orbit of starting point close to the stable manifold of $Z^a$, blue curve is a forward orbit of a starting point close to the unstable manifold of $Z^b$. The curves stay either (a) close to each other for all times, (b) close near the subspace $\{ y_2=0 \}$, (c) close near $\{ y_3=0 \}$, (d) or are faraway from each other.}
\end{figure}

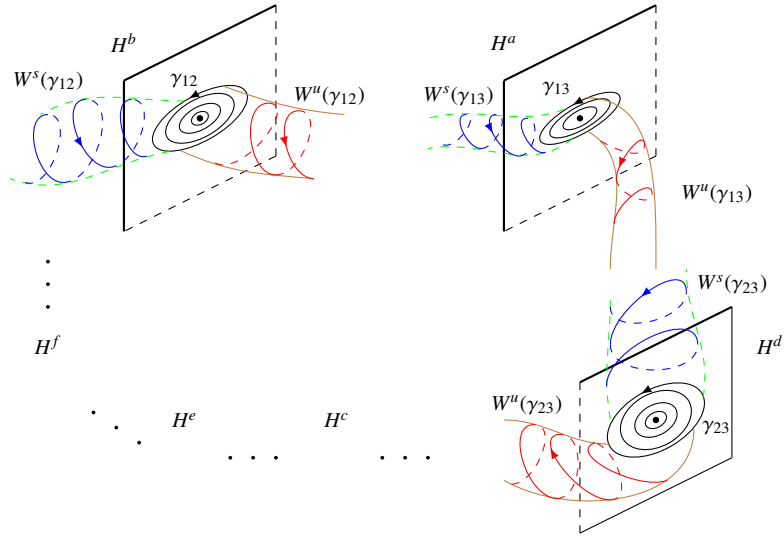
\begin{figure}
\centering
\begin{tikzpicture}[xscale=1,>=latex]
\coordinate (Za) at (-4,4);
\coordinate (Za1) at (-5,2.5);
\coordinate (Za2) at (-5,4.5);
\coordinate (Za3) at (-3,5.5);
\coordinate (Za4) at (-3,3.5);
\coordinate (Za5) at (-3.7,4.43);
\coordinate (Za6) at (-3.2,4.2);
\coordinate (Za7) at (-2.5,4.1);
\coordinate (Za8) at (-2.1,4.05);
\coordinate (Za51) at (-4.25,3.57);
\coordinate (Za61) at (-4.1,3.3);
\coordinate (Za71) at (-3.4,3.2);
\coordinate (Za81) at (-2.5,3.2);
\coordinate (Za9) at (-4.2,4.23);
\coordinate (Za10) at (-3.2,4.2);
\coordinate (Za11) at (-2.5,4.1);
\coordinate (Za12) at (-6.1,4.05);
\coordinate (Za91) at (-4.4,3.56);
\coordinate (Za101) at (-4.1,3.3);
\coordinate (Za111) at (-3.4,3.2);
\coordinate (Za121) at (-6.5,3.2);

\node at (-5,5) {\scriptsize $H^b$};
\filldraw
(Za) circle (1pt);
\node at (-4.2,4.5) {\scriptsize $\gamma_{12}$};
\node at (-6,4.5) {\scriptsize $W^s(\gamma_{12})$};
\node at (-2.3,4.3) {\scriptsize $W^u(\gamma_{12})$};
\draw[rotate=-60,decoration={markings, mark=at position 0.5 with {\arrow{>}}}, postaction={decorate}] (Za) ellipse (0.3 and 0.7);
\draw[thin,rotate=-60] (Za) ellipse (0.25 and 0.5);
\draw[rotate=-60] (Za) ellipse (0.15 and 0.3);
\draw[rotate=-60] (Za) ellipse (0.08 and 0.12);
\draw[thick] (Za1) -- (Za2);
\draw[thick] (Za2) -- (Za3);
\draw[dashed] (Za3) -- (Za4);
\draw[dashed] (Za4) -- (Za1);
\draw[brown] (Za5) to[out=-20,in=175] (Za8);
\draw[brown] (Za51) to[out=-25,in=175] (Za81);
\draw[green,dashed] (Za9) to[out=180,in=30] (Za12);
\draw[green,dashed] (Za91) to[out=190,in=-20] (Za121);
\draw[blue,dashed] (-6.3,3.15) to[out=-20,in=0] (-6.0,4.05);
\draw[blue] (-6.0,4.05) to[out=180,in=180] (-5.8,3.15); 
\draw[blue,dashed] (-5.8,3.15) to[out=0,in=0] (-5.5,4.25);
\draw[blue,decoration={markings, mark=at position 0.63 with {\arrow{>}}}, postaction={decorate}] (-5.5,4.25) to[out=180,in=180] (-5.2, 3.35);
\draw[blue,dashed] (-5.2, 3.35) to[out=0,in=0] (-4.9,4.25);
\draw[blue] (-4.9,4.25) to[out=180,in=200] (-4.7, 3.5); 
\draw[red,dashed] (-3.8,3.41) to[out=-30,in=0] (-3.1,4.2);
\draw[red] (-3.1,4.2) to[out=180,in=180] (-3.0,3.25);%
\draw[red,dashed] (-3.0,3.25) to[out=0,in=0] (-2.7,4.1);
\draw[red,decoration={markings, mark=at position 0.5 with {\arrow{>}}}, postaction={decorate}] (-2.7,4.1) to[out=180,in=180] (-2.5, 3.2); 

\coordinate (Zb) at (1,4);
\coordinate (Zb1) at (0,2.5);
\coordinate (Zb2) at (0,4.5);
\coordinate (Zb3) at (2,5.5);
\coordinate (Zb4) at (2,3.5);
\coordinate (Zb5) at (0.8,4.12);
\coordinate (Zb8) at (-1,4);
\coordinate (Zb51) at (0.9,3.73);
\coordinate (Zb81) at (-1,3.6);
\coordinate (Zb9) at (1.1,4.27);
\coordinate (Zb12) at (2,2);
\coordinate (Zb91) at (1.1,3.83);
\coordinate (Zb121) at (1.4,2);

\node at (0,5) {\scriptsize $H^a$};
\filldraw
(Zb) circle (1pt);
\node at (0.7,4.4) {\scriptsize $\gamma_{13}$};
\node at (-0.6,4.3) {\scriptsize $W^s(\gamma_{13})$};
\node at (2.8,3) {\scriptsize $W^u(\gamma_{13})$};
\draw[rotate=-60,decoration={markings, mark=at position 0.5 with {\arrow{>}}}, postaction={decorate}] (Zb) ellipse (0.2 and 0.6);
\draw[rotate=-60] (Zb) ellipse (0.15 and 0.45);
\draw[rotate=-60] (Zb) ellipse (0.1 and 0.25);
\draw[thick] (Zb1) -- (Zb2);
\draw[thick] (Zb2) -- (Zb3);
\draw[dashed] (Zb3) -- (Zb4);
\draw[dashed] (Zb4) -- (Zb1);
\draw[green,dashed] (Zb5) to[out=200,in=15] (Zb8);
\draw[green,dashed] (Zb51) to[out=220,in=10] (Zb81);
\draw[brown] (Zb9) to[out=0,in=90] (Zb12);
\draw[brown] (Zb91) to[out=-30,in=90] (Zb121);
\draw[blue,dashed] (-0.8,3.65) to[out=-20,in=0] (-0.5,4.05);
\draw[blue] (-0.5,4.05) to[out=180,in=180] (-0.3,3.6); 
\draw[blue,dashed] (-0.3,3.6) to[out=0,in=0] (-0.1,4.05);
\draw[blue,decoration={markings, mark=at position 0.4 with {\arrow{>}}}, postaction={decorate}] (-0.1,4.05) to[out=180,in=180] (0.2, 3.5);
\draw[blue,dashed] (0.2, 3.5) to[out=0,in=0] (0.3,4);
\draw[blue] (0.3,4) to[out=180,in=200] (0.55, 3.55); 
\draw[red,dashed] (1.27,3.7) to[out=-30,in=-90] (1.85,3.65);
\draw[red, decoration={markings, mark=at position 0.8 with {\arrow{>}}}, postaction={decorate}] (1.85,3.65) to[out=90,in=90] (1.48,3.2);%
\draw[red,dashed] (1.48,3.2) to[out=-90,in=-90] (1.95,3.0);
\draw[red] (1.95,3.0) to[out=90,in=90] (1.45,2.6);%

\coordinate (Zc) at (2,0);
\coordinate (Zc1) at (1,-1.5);
\coordinate (Zc2) at (1,0.5);
\coordinate (Zc3) at (3,1.5);
\coordinate (Zc4) at (3,-0.5);
\coordinate (Zc5) at (1.4,0.01);
\coordinate (Zc8) at (1.4,1.6);
\coordinate (Zc51) at (2.62,0.3);
\coordinate (Zc81) at (2.4,2);
\coordinate (Zc9) at (1.4,-0.32);
\coordinate (Zc12) at (0,0);
\coordinate (Zc91) at (2.5,-0.12);
\coordinate (Zc121) at (0,-0.8);

\node at (3.5,1) {\scriptsize $H^d$};
\node at (2.8,-0.1) {\scriptsize $\gamma_{23}$};
\node at (3,1.8) {\scriptsize $W^s(\gamma_{23})$};
\node at (0.3,0.2) {\scriptsize $W^u(\gamma_{23})$};
\filldraw
(Zc) circle (1pt);
\draw[rotate=-60,decoration={markings, mark=at position 0.5 with {\arrow{>}}}, postaction={decorate}] (Zc) ellipse (0.4 and 0.7);
\draw[rotate=-60] (Zc) ellipse (0.35 and 0.5);
\draw[rotate=-60] (Zc) ellipse (0.2 and 0.3);
\draw[rotate=-60] (Zc) ellipse (0.1 and 0.15);
\draw[dashed] (Zc1) -- (Zc2);
\draw[thick] (Zc2) -- (Zc3);
\draw (Zc3) -- (Zc4);
\draw (Zc4) -- (Zc1);
\draw[green,dashed] (Zc5) to[out=100,in=-90] (Zc8);
\draw[green,dashed] (Zc51) to[out=800,in=-90] (Zc81);
\draw[brown] (Zc9) to[out=180,in=0] (Zc12);
\draw[brown] (Zc91) to[out=-90,in=-30] (Zc121);

\draw[blue,dashed] (1.4,1.5) to[out=-90,in=-90] (2.4,1.7);
\draw[blue,decoration={markings, mark=at position 0.5 with {\arrow{>}}}, postaction={decorate}] (2.4,1.7) to[out=90,in=90] (1.4,1); 
\draw[blue,dashed] (1.4,1) to[out=-90,in=-90] (2.55,1);
\draw[blue] (2.55,1) to[out=90,in=90] (1.35,0.45); 
\draw[red,dashed] (0.1,-0.85) to[out=0,in=0] (0.4,-0.06);
\draw[red] (0.4,-0.06) to[out=180,in=180] (0.7,-1.04); 
\draw[red,dashed] (0.7,-1.04) to[out=0,in=0] (0.75,-0.19);
\draw[red,decoration={markings, mark=at position 0.3 with {\arrow{<}}}, postaction={decorate}] (0.75,-0.19) to[out=180,in=180] (1.4,-1.06); 
\draw[red,dashed] (1.4,-1.06) to[out=0,in=-20] (1.2,-0.3);
\draw[red] (1.2,-0.3) to[out=200,in=200] (2.15,-0.8);

\coordinate (C21) at (-2.2,0);
\node at (C21) {\scriptsize $H^c$};
\coordinate (C31) at (-4.2,0);
\node at (C31) {\scriptsize $H^e$};
\coordinate (C32) at (-6,1);
\node at (C32) {\scriptsize $H^f$};
\filldraw
(-1,-0.5) circle (0.5pt);
\filldraw
(-1.3,-0.5) circle (0.5pt);
\filldraw
(-1.6,-0.5) circle (0.5pt);
\filldraw
(-3,-0.5) circle (0.5pt);
\filldraw
(-3.3,-0.5) circle (0.5pt);
\filldraw
(-3.6,-0.5) circle (0.5pt);
\filldraw
(-4.8,-0.3) circle (0.5pt);
\filldraw
(-5.1,-0.1) circle (0.5pt);
\filldraw
(-5.4,0.1) circle (0.5pt);
\filldraw
(-6,1.5) circle (0.5pt);
\filldraw
(-6,1.8) circle (0.5pt);
\filldraw
(-6,2.1) circle (0.5pt);
\end{tikzpicture}
\caption{Six two-dimensional invariant squares described in the propositions (\ref{6addEq}) and (\ref{non-transverse}). Red/blue curves relate to the fibers of the unstable/stable manifolds of the periodic orbits lying within invariant squares, respectively. }
\label{fig:homoclinicChannel}
\end{figure}

\subsection{Dynamics in the neighbourhood of the non-transverse homoclinic channel in the Rock-Scissors-Paper game}\label{RSPsub}
In this subsection we describe the non-transverse homoclinic channel appearing in the Rock-Scissors-Paper game and as well investigate the dynamics in its neighbourhood.
\newline
To start with, note that from the latter remark \ref{non-transverse}, it follows that there exists a homoclinic channel consisting of invariant subspaces $H^a$,...,$H^f$ and its invariant manifolds, which are connected in sequence $H^a \to H^d \to H^c \to  H^e \to  H^f \to  H^b \to  H^a$. Here $H^a \to H^d$ means that $W^u(H^a)=W^s(H^d)$ (analogously for other connections). 
Note that this homoclinic channel, which we denote by $I$, is non-transverse, since corresponding stable/unstable manifolds coincide.
\newline 
We want to investigate the dynamics near the channel $I$. 
\begin{rmq}
Since for any point in $I$, its omega limit set is the subset of one of the invariant subspaces $H^a$,...,$H^f$ (the same holds for alpha limit set), our idea would be to reduce the study of the dynamics in the vicinity of the channel $I$ to the composition of $6$ scattering maps 
$H^a \mapsto H^d$, $H^d \mapsto H^c$, $H^c \mapsto  H^e$, $H^e \mapsto  H^f$, $H^f \mapsto  H^b$, $H^b \mapsto  H^a$. For this purpose, it would be necessary to establish (infinite) shadowing of the composition of the scattering maps. Since the homoclinic channel $I$ is non-transverse, we cannot follow the approach presented in \cite{Scatt}, \cite{ShadowingLemma}, \cite{Diffusion1}, \cite{Cov1}. Note however that these $6$ scattering maps are well defined since these six pieces of homoclinic channel are in fact invariant subspaces, either $\{ x_i = 0 \}$ or $ \{y_i = 0 \}$.  Instead, in the remaining part of this section we want to check whether the channel $I$ is locally attracting, i.e. if there exists a neighbourhood of the channel $I$ so that each orbit from this neighbourhood accumulates on $I$.
\end{rmq}
Note that due to the $\mathbb Z _3$ symmetry of the Rock-Scissors-Paper model it suffices to investigate only subspaces $H^b = \{ x_1=0, \ y_2 = 0 \}$ and $H^c = \{ x_2=0, \ y_1 = 0 \}$. 
\newline
For $H^b$, we consider the $(x_1,y_2)$ part of the vector field of the original system 
\begin{equation*}
\begin{cases}
\dot{x_1}=x_1 [(Ay)_1-x^T Ay] 
\\ 
\dot{x_2}=x_2 [(Ay)_2-x^T Ay]
\\ 
\dot{y_1}=y_1[(Bx)_1-y^T Bx] 
\\
\dot{y_2}=y_2[(Bx)_2-y^T Bx]
\end{cases}
\end{equation*}
with substitutions $x_3=1-x_1-x_2, y_3=1-y_1-y_2$ and constraints $x_1,x_2 \geq 0, x_1+x_2 \leq 1$, $y_1,y_2 \geq 0, y_1+y_2 \leq 1$.
\newline
That is  
\begin{equation*}
\begin{cases}
\dot{x_1}=x_1 [(Ay)_1-x^T Ay] 
\\
\dot{y_2}=y_2[(Bx)_2-y^T Bx]
\end{cases}
\end{equation*}
and corresponding linearized vector field at $x_1 = 0, \ y_2=0$, with $x_3=1-x_2, y_3=1-y_1$.
\begin{equation*}
V(x_2,y_1)= \left( \begin{array}{c}
V_1(x_2,y_1) \\
V_2(x_2,y_1) 
\end{array} \right)
:=\left( \begin{array}{c}
(Ay)_1-x^T Ay \\
(Bx)_2-y^T Bx 
\end{array} \right) 
\end{equation*}
If $\gamma$ is a periodic solution of the (original) system reduced to the subspace $H^b$ (\ref{x2y1RSP} below), with substitutions $x_3=1-x_2, y_3=1-y_1$.
\begin{equation}\label{x2y1RSP}
\begin{cases}
\dot{x_2}=x_2 [(Ay)_2-x^T Ay]
\\ 
\dot{y_1}=y_1[(Bx)_1-y^T Bx] 
\end{cases}
\end{equation}
then the rate of attraction/repulsion to/from the periodic orbit $\gamma \subset H^b$ is a number:
$$\int _{\gamma}{\big(  V_1(x_2,y_1) + V_2(x_2,y_1) \big)} = \int _{\gamma}{\big( (Ay)_1-x^T Ay + (Bx)_2-y^T Bx \big)}$$
The corresponding rates for hyperplane $H^c$ can be computed analogously.

\begin{figure}[!htb]
\minipage{0.40\textwidth}
  \includegraphics[width=\linewidth]{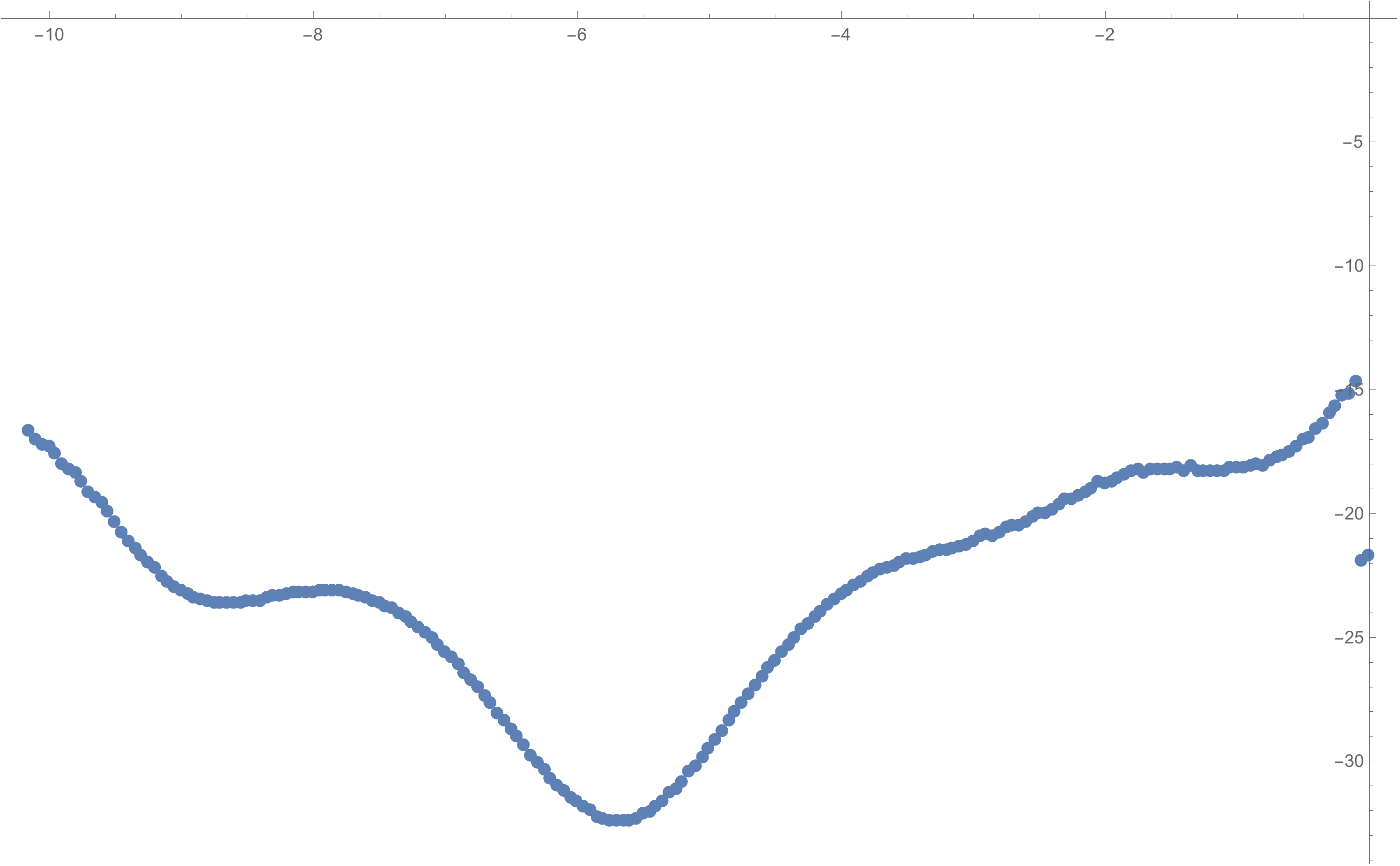}
	{\tiny (a)	}
\endminipage\hfill
\minipage{0.40\textwidth}
  \includegraphics[width=\linewidth]{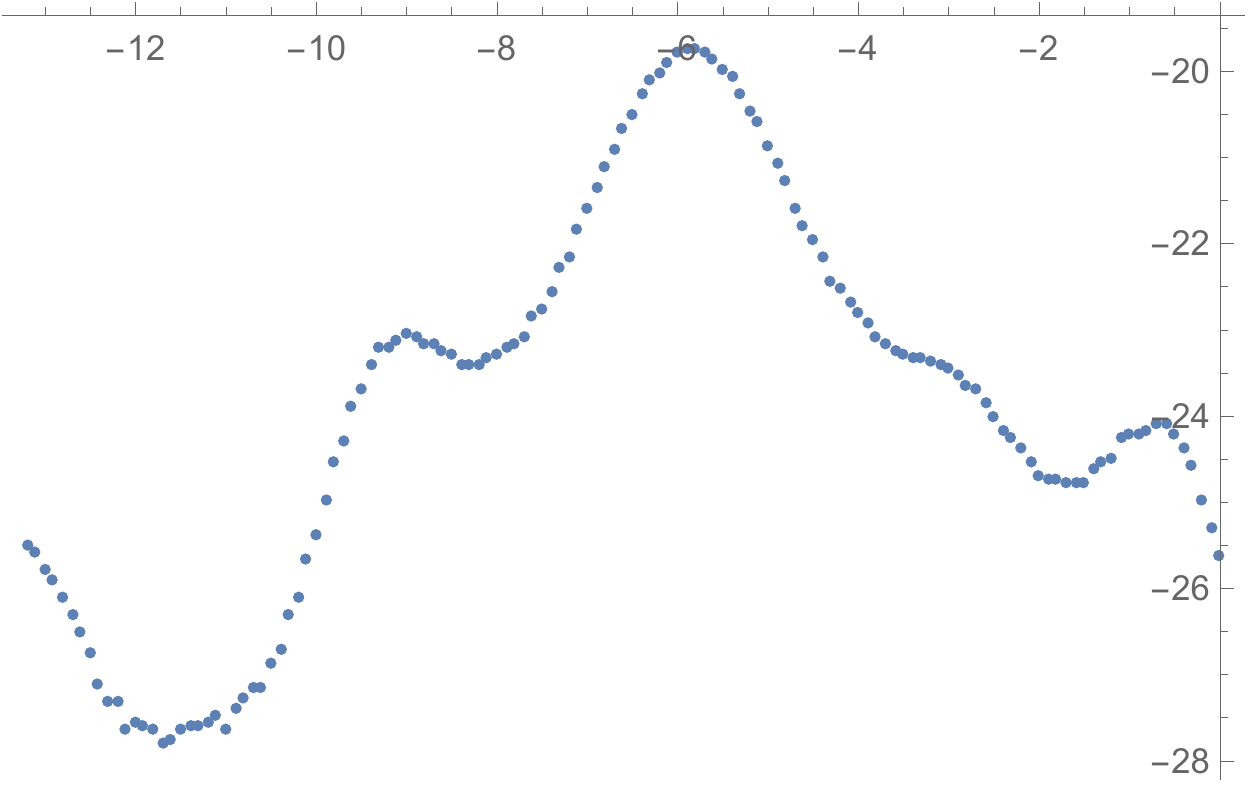}
	{\tiny (b)	}
\endminipage\hfill
\caption{Image of two different periodic orbits $\gamma_1, \gamma_2$ from $H^b$ under the same scattering map $H^b \rightarrow H^a$. 
The horizontal axis parameterizes the time, ranging from $0$ to the period of the orbit, and the vertical axis parameterizes periodic orbits in $H^a$.
\newline
This figure suggests that the unstable manifold  of $\gamma_1$ and of $\gamma_2$ intersect transversally (within the invariant subspace $\{ x_1 =0\}$) stable manifolds of the nontrivial family of periodic orbits from $H^a$.
}
\label{fig:scattMapImage}
\end{figure}

\begin{figure}
\centering
\includegraphics[width=120mm]{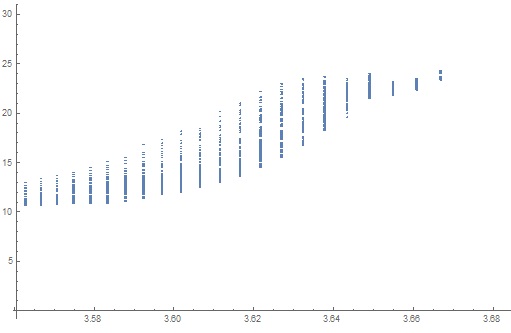}
\caption{Picture presenting a joint image of periodic orbits under the scattering map. Horizontal axis parameterizes (by energy function) periodic orbits in $H^b$, vertical axis parameterizes periodic orbits in $H^a$}
\label{fig:blender1}
\end{figure}

\begin{figure}[!htb]
\minipage{0.45\textwidth}
  \includegraphics[width=\linewidth]{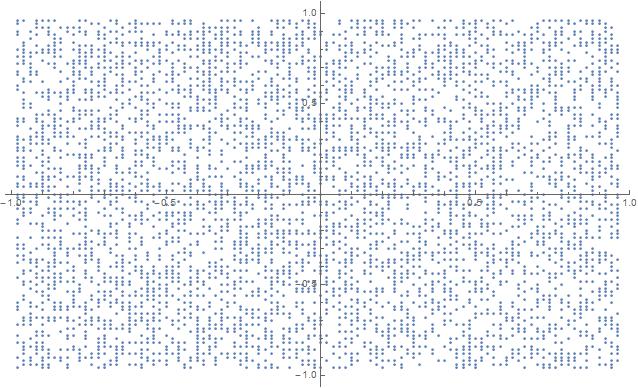}
	{\small (a)  Parameter values for which at least 1\% of the points drawn close to the one of the center manifolds visit in sequence neighbourhoods of 60 consecutive center manifolds. 
		\newline
		}
\endminipage\hfill
\minipage{0.45\textwidth}
  \includegraphics[width=\linewidth]{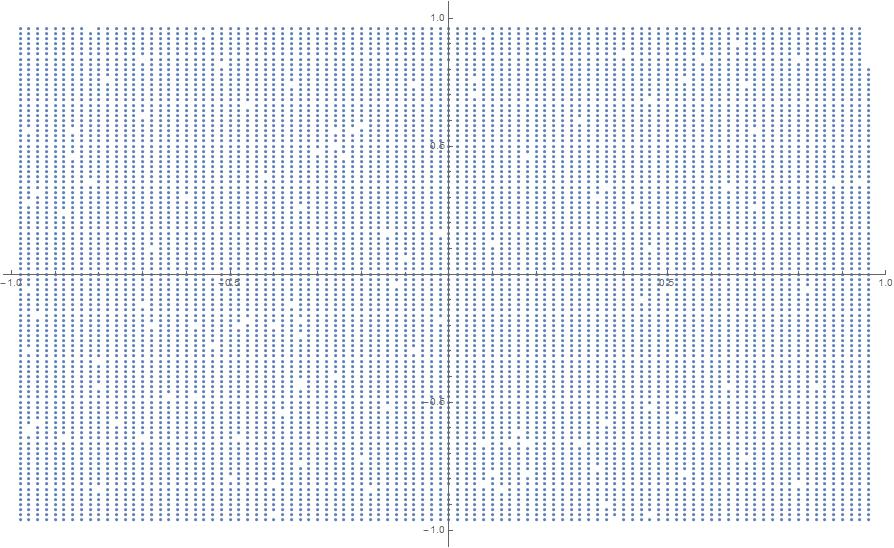}
	{\small (b)  Parameter values for which at least 1\% of the points drawn close to the one of the center manifolds visit in sequence neighbourhoods of 12 consecutive center manifolds, i.e. comes  back two times to the neighbourhood of the initial center manifold. 
	}
\endminipage\hfill
\minipage{0.45\textwidth}%
  \includegraphics[width=\linewidth]{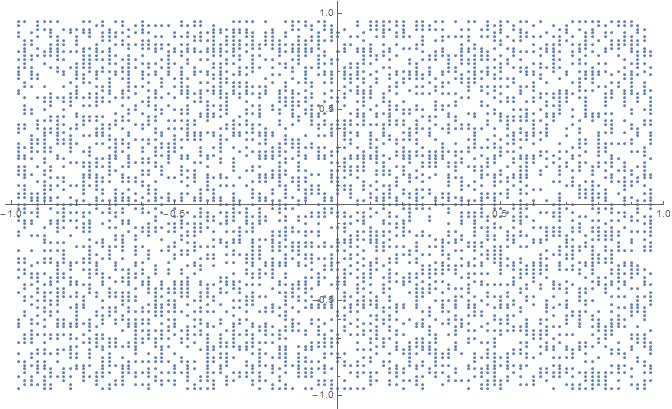}
	{\small (c)  Parameter values for which at least 10\% of the points drawn close to the one of the center manifolds visit in sequence neighbourhoods of 6 consecutive center manifolds, i.e. comes back to the neighbourhood of the initial center manifold.  }
\endminipage\hfill
\minipage{0.45\textwidth}%
  \includegraphics[width=\linewidth]{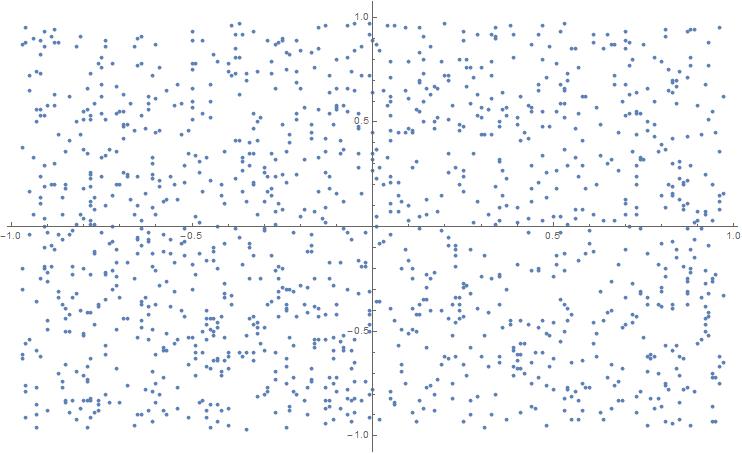}
	{\small (d)  Parameter values for which at least 25\% of the points drawn close to the one of the center manifolds visit in sequence neighbourhoods of 6 consecutive center manifolds, i.e. comes back to the neighbourhood of the initial center manifold. }
\endminipage\hfill
\caption{Numerical investigation of the repulsion from the homoclinic channel in the Rock-Scissors-Paper game}
\label{fig:RSPabcd}
\end{figure}

\begin{rmq}
For all parameter values $\epsilon_X, \epsilon_Y$, for which we computed these rates of attraction/repulsion to/from periodic orbits, we do not find any reason for the total sum of these rates, for the whole loop consisting of six hyperplanes $H^a$, $H^b$, $H^c$, $H^d$, $H^e$, $H^f$, to be negative. This explains the figures (\ref{fig:RSPabcd}), where we investigate numerically shadowing of composition of $6$ scattering maps corresponding to $6$ pieces of the non-transverse homoclinic channel. 
\end{rmq}
In conclusion, in finite time the orbit of an initial point being close to the unstable manifold (but not belonging to it) of any of these invariant $2$ dimensional hyperplanes ($H^a$,...,$H^f$), will eventually stop shadowing the sequence of hyperplanes and leave the neighbourhood of this non-transverse homoclinic channel $I$. 
\begin{cor}
The homoclinic channel $I$ is repelling and so the dynamics near it cannot be fully described by the composition of the aforementioned 6 scattering maps.
\end{cor}
In fact, the calculated ratios of attraction/repulsion to the orbits tell us that when considering the flow backward in time, all of the orbits from some neighbourhood of the channel $I$ accumulate on $I$.
\begin{rmq}
In the section \ref{sec:toyModel} we will consider a toy-model motivated by this example of a non-transverse homoclinic channel $I$ from Rock-Scissors-Paper. We will identify invariant subspaces $H^a$,...,$H^f$ with one center manifold $D$ foliated with periodic orbits and assume that the stable/unstable manifold of this center manifold $D$ coincide. However we will assume that this homoclinic channel is locally attracting, so the toy-model will resemble a homoclinic channel $I$ from the Rock-Scissors-Paper game with reversed time.
\end{rmq}
\subsection{Discussion}
What we were actually interested in, is the form of the scattering maps between the invariant squares $H^a,..., H^f$ and what information they provide about the dynamics near the non-transverse homoclinic channel $I$ consisting of invariant squares and their stable/unstable manifolds (see figure \ref{fig:homoclinicChannel}).
\newline
In the figure (\ref{fig:scattMapImage}), there are two examples of the images of periodic orbits (from invariant squares) under the scattering map. 
\newline
Figure (\ref{fig:blender1}) presents joint image of the family of parameterized periodic orbits from one invariant square under the scattering map. See the captions under the figures for details and explanations.
\newline
However, the conclusion of this section is that although the images of the scattering maps look very complicated, due to lack of the attraction to the homoclinic channel and lack of infinite shadowing of the scattering maps (compare with \cite{Scatt}, \cite{Diffusion1}, \cite{ShadowingLemma} and \cite{Cov1}, where the transversality condition for the homoclinic channel is essential), one cannot reduce the study of the behaviour near this non-transverse homoclinic channel to the study of the scattering maps between the invariant squares.
\newline
All of the numerical computations and simulations were performed with the usage of the tools implemented in CAPD Library \cite{CAPD}.

\section{The toy-model}\label{sec:toyModel}
\subsection{Normal form of the ODE near the center manifold}\label{nf1}
Motivated by the example of a non-transverse homoclinic channel $I$ in the Rock-Scissors-Paper game, we consider in this section the following ODE in $\mathbb R^4$:
\begin{equation}\label{1rown}
\dot{x} = F(x)
\end{equation}
with $F$ - real analytic, $F(0)=0$ and the point spectrum $\sigma(dF(0)) = \{\lambda_1, \lambda_2, \lambda_3, \lambda_4 \}$, where $\lambda_1>0$, $\lambda_2<0$, $|\lambda_3|=1$, $\lambda_3=\bar{\lambda}_4$.
\newline
Assume that the fixed point $x=0$ possesses a $2$-dimensional compact center manifold $W^c(0)$ on which the flow is periodic (i.e. the center manifold is foliated with periodic orbits). After straightening $W^c(0)$, so that it becomes a $2$-dimensional disc $D$ and introducing the polar coordinates $(r,\phi) \in [0,1) \times [0,1)$ on it, we perform the change of the coordinates $(x_1,x_2,x_3,x_4) \mapsto (v_1, v_2, r, \phi)$ so that we can rewrite (\ref{1rown}) in the following form: 
\begin{equation}\label{2rown}
\begin{cases}
\dot{v_1}= f_1(v_1,v_2,r,\phi)
\\
\dot{v_2}= f_2(v_1,v_2,r,\phi)
\\
\dot{r} = f_3(v_1,v_2,r,\phi)
\\
\dot{\phi} = f_4(v_1,v_2,r,\phi)
\end{cases}
\end{equation}
with $\frac{\partial f_1}{\partial v_1} (0,0,0,0)= \lambda_1$, $\frac{\partial f_2}{\partial v_2} (0,0,0,0)= \lambda_2$, $\frac{\partial f_1}{\partial v_2} (0,0,0,0)= \frac{\partial f_2}{\partial v_1} (0,0,0,0)= 0$. 
\newline
Moreover assume that $3$ - dimensional stable and unstable manifolds (i.e. $W^s(D)$, $W^u(D)$) of the center manifold $D=W^c(0)$ coincide.
\begin{rmq}
In particular their intersection is non-transverse and forms a non-transverse homoclinic channel which we denote by $I$. Because of the lack of the transversality condition, $I$ does not fit into the definition of a homoclinic channel from \cite{Scatt}. We emphasize it by calling $I$ a non-transverse homoclinic channel.
\end{rmq}

\begin{figure}
\centering
\begin{tikzpicture}[xscale=1,>=latex]
\coordinate (Za) at (1,-1); 

\node at (-2.4,-2) {\scriptsize $D$};
\node at (0.65,-1.05) {\scriptsize $0$};
\filldraw
(Za) circle (1pt);
\node at (0.5,2.5) {\scriptsize $W^u(D)$};
\node at (0.8,-3.8) {\scriptsize $W^s(D)$};
\draw[dashed,rotate=-70,decoration={markings, mark=at position 0.05 with {\arrow{>}}}, postaction={decorate}] (Za) ellipse (0.3 and 1);
\draw[dashed,rotate=-68,decoration={markings, mark=at position 0.95 with {\arrow{>}}}, postaction={decorate}] (Za) ellipse (0.9 and 1.8);
\draw[rotate=-67,decoration={markings, mark=at position 0.84 with {\arrow{>}}}, postaction={decorate}] (Za) ellipse (1.6 and 3.2);

\draw[brown] (-1.95,-1.5) to[out=80,in=160] (4.25,2.5);
\draw[brown] (4.25,2.5) to[out=-20,in=90] (8.25,0);
\draw[brown] (4,0) to[out=90,in=180] (5.5,1.5);
\draw[brown] (5.5,1.5) to[out=0,in=90] (6.5,0);

\draw[dashed,brown] (-0.7,-1.5) to[out=80,in=170] (4.2,2.2);
\draw[dashed,brown] (2.62,-0.28) to[out=85,in=-160] (4,1.5);
\draw[red,decoration={markings, mark=at position 0.5 with {\arrow{>}}}, postaction={decorate}] (2,1.9) to[out=40,in=-160] (4,1.5);
\draw[red,dashed] (2,1.9) to[out=200,in=190] (3.15,1);
\draw[red,decoration={markings, mark=at position 0.4 with {\arrow{<}}}, postaction={decorate}] (3.15,1) to[out=140,in=-70] (0.6,1);
\draw[red,dashed] (0.6,1) to[out=40,in=120] (2.75,0.3);
\draw[red,decoration={markings, mark=at position 0.2 with {\arrow{<}}}, postaction={decorate}] (2.75,0.3) to[out=-100,in=-40] (-0.4,-0.4);

\draw[green] (-2,-2.1) to[out=-80,in=180] (3.25,-4);
\draw[green] (3.6,-1.07) to[out=-75,in=160] (4.5,-2);
\draw[green] (3.25,-4) to[out=0,in=-90] (8.25,0);
\draw[green] (4.5,-2) to[out=-20,in=-90] (6.5,0);

\draw[dashed,green] (0.05,-1.3) to[out=-80,in=190] (4.4,-3.5);
\draw[dashed,green] (1.9,-0.8) to[out=-85,in=160] (4,-2.5);

\draw[blue,decoration={markings, mark=at position 0.4 with {\arrow{<}}}, postaction={decorate}] (0.25,-1.9) to[out=-60,in=240] (2.45,-1.8);
\draw[blue,dashed] (2.45,-1.8) to[out=160,in=100] (0.8,-2.7);
\draw[blue,decoration={markings, mark=at position 0.4 with {\arrow{<}}}, postaction={decorate}] (0.8,-2.7) to[out=-50,in=-70] (3,-2.15);
\draw[blue,dashed] (3,-2.15) to[out=150,in=120] (2.5,-3.5);
\end{tikzpicture}
\caption{2 - dimensional center manifold $D$, foliated with periodic orbits, with its coinciding stable and unstable manifolds $W^s(D)=W^u(D)$}
\label{fig:toyModel}
\end{figure}
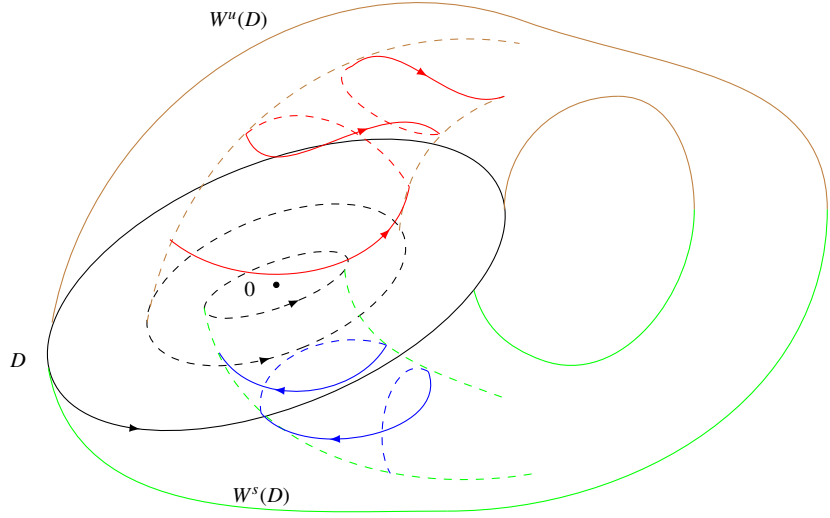
After expanding $f_1$, $f_2$, $f_3$, $f_4$ in terms of $(v_1,v_2)$ in the vicinity of $(0,0)$, for any $(r,\phi) \in [0,1] \times [0,1)$ and straightening local unstable and stable manifolds of $D=W^c(0)$ to be locally the $3$-dimensional tubes $(v_1,0,r,\phi)$ and $(0,v_2,r,\phi)$ we arrive at the simplified system of equations given by the analytic vector field:
\begin{equation}\label{3rown}
\begin{cases}
\dot{v_1}= p(r,\phi) \cdot v_1  + p_{00}(v_1,v_2,r,\phi) \cdot v_1
\\
\dot{v_2}= \varsigma(r,\phi) \cdot v_2 +  \varsigma_{00}(v_1,v_2,r,\phi) \cdot v_2
\\
\dot{r} = r_{00}(r,\phi) + v_1 v_2 \cdot r_{0}(v_1,v_2,r,\phi) 
\\
\dot{\phi} = \omega_{00}(r,\phi) + v_1 v_2 \cdot \omega_{0}(v_1,v_2,r,\phi) 
\end{cases}
\end{equation}
with $r_{00}, \omega_{00}, r_0, \omega_0, p, p_{00}, \varsigma, \varsigma_{00}$ being $1$-periodic with respect to $\phi$.
\begin{ass}
Since $v_1$, $v_2$ are respectively unstable and stable directions, then $p(r,\phi)>0$ and $\varsigma(r,\phi)<0$, for all $r$, $phi$. Moreover we assume that $p(r,\phi) + \varsigma(r,\phi)<0$, which will eventually result in channel $I$ being locally attracting, i.e. orbits starting near the channel $I$ will accumulate on $I$.
\end{ass}
Until the end of this section we will perform analytic transformations of the coordinates leading us to the analytic and convergent normal form of the ODE describing this model near the center manifold:
\begin{equation}\label{normFormODE}
\begin{cases}
\dot{v_1}= p(r) \cdot v_1  + p_0(v_1,v_2,r,\phi) \cdot {v_1}^2 v_2
\\
\dot{v_2}= \varsigma(r) \cdot v_2 +  \varsigma_0(v_1,v_2,r,\phi) \cdot v_1 {v_2}^2
\\
\dot{r} = v_1 v_2 \cdot r_{0}(v_1,v_2,r,\phi) 
\\
\dot{\phi} = \omega(r) + v_1 v_2 \cdot \omega_{0}(v_1,v_2,r,\phi) 
\end{cases}
\end{equation}
with $p_0$, $\varsigma_0$, $r_0$, $\omega_0$ being $1$ - periodic in $\phi$.
\newline
We start with simplifications of $(r,\phi)$ coordinates in (\ref{3rown}), for this we need the following lemma about the action-angle coordinates in the integrable Hamiltonian system:
\begin{lemma}\label{coordDisc}
Under the above assumptions, there exists an analytic and convergent change of coordinates that transforms system
\begin{equation}\label{}
\begin{cases}
\dot{r} = r_{00}(r,\phi) 
\\
\dot{\phi} = \omega_{00}(r,\phi) 
\end{cases}
\end{equation}
into the action-angle coordinates
\begin{equation}\label{}
\begin{cases}
\dot{\xi} = 0 
\\
\dot{\zeta} = \omega(\xi) 
\end{cases}
\end{equation}
with $0 \leq r < 1$, $0  \leq \xi <  1 $ and $\omega(r):= \int_{0}^{1}{\omega_{00}(r, \phi') d \phi'}$
\end{lemma}
\begin{proof}
Let us write
\newline
$$\xi(t):=r(t) \cdot \Psi(t), \ \ \ \  \zeta(t):= r(t) \cdot \epsilon(t)$$
Then
$$0= \dot{\xi}(t) = \Psi(\phi) \cdot \dot{r} + \frac{d \Psi}{d t} \cdot r = \Psi(\phi) \cdot r_{00}(r,\phi)  + \frac{d \Psi}{d t} \cdot r $$
 and 
$$\omega(\xi)= \dot{\zeta}(t) = \epsilon(\phi) \cdot \dot{r} + \frac{d \epsilon}{d t} \cdot r = \epsilon(\phi) \cdot r_{00}(r,\phi) + \frac{d \epsilon}{d t} \cdot r $$
Hence 
$$\Psi(t) = exp \Bigg(- \int_{0}^{t}{\frac{r_{00}(r(s),\phi(s))}{r(s)} \ d s} \Bigg)$$ 
under assumption $\Psi(0)=1$. In the same way we find that
$$\epsilon(t) =  exp \Bigg(- \int_{0}^{t}{\frac{r_{00}(r(s),\phi(s))}{r(s)} \ d s} \Bigg) \cdot \int_{0}^{t} \omega(r(t')) \cdot exp \Bigg( \int_{0}^{t'}{\frac{r_{00}(r(s),\phi(s))}{r(s)} \ d s} \Bigg) \ d t' $$
\end{proof}

Due to the above lemma (\ref{coordDisc}), we can assume that the system (\ref{3rown}) is governed by the equations:
\begin{equation}\label{4rown}
\begin{cases}
\dot{v_1}= p(r,\phi) \cdot v_1  + p_{00}(v_1,v_2,r,\phi) \cdot v_1
\\
\dot{v_2}= \varsigma(r,\phi) \cdot v_2 +  \varsigma_{00}(v_1,v_2,r,\phi) \cdot v_2
\\
\dot{r} = v_1 v_2 \cdot r_{0}(v_1,v_2,r,\phi) 
\\
\dot{\phi} = \omega(r) + v_1 v_2 \cdot \omega_{0}(v_1,v_2,r,\phi) 
\end{cases}
\end{equation}
In order to reduce the above system \ref{4rown} by coordinates transformation into \ref{normFormODE}, we follow the steps of reduction as in \cite{Shil1}. Let us introduce:
\newline
\begin{equation}
\begin{cases}
p_1(r,\phi,v_1) :=p_{00}(v_1,0,r,\phi)
\\
p_2(r,\phi,v_2) :=p_{00}(0,v_2,r,\phi)
\\
p_3(r,\phi,v_1,v_2) :=p_{00}(v_1,v_2,r,\phi) - p_1(r,\phi,v_1) - p_2(r,\phi,v_2)
\\
\varsigma_1(r,\phi,v_1) :=\varsigma_{00}(v_1,0,r,\phi)
\\ 
\varsigma_2(r,\phi,v_2) :=\varsigma_{00}(0,v_2,r,\phi)
\\
\varsigma_3(r,\phi,v_1,v_2) :=\varsigma_{00}(v_1,v_2,r,\phi) - \varsigma_1(r,\phi,v_1) - \varsigma_2(r,\phi,v_2)
\end{cases}
\end{equation}

In the system (\ref{4rown}) we want to get rid of the terms $p_1, \ p_2, \ \varsigma_1, \ \varsigma_2$, by subsequent analytic and convergent changes of coordinates. Next four propositions will consequently deal with the terms $p_2$, then $\varsigma_1$, $p_1$, $\varsigma_2$ (see \cite{Shil1}, \cite{Shil2}).

\begin{prop}
Killing the term $p_2$ with substitution $\xi_1 := v_1 + h_1(r,\phi, v_2) \cdot v_1$
\end{prop}
\begin{proof}
We have
$$\dot{\xi}_1 = \dot{v}_1 \cdot (1 + h) + v_1 \cdot \Big(\frac{\partial h_1}{\partial r} \dot{r} + \frac{\partial h_1}{\partial \phi} \dot{\phi} + \frac{\partial h_1}{\partial v_2} \dot{v}_2 \Big) = $$
$$= v_1 \cdot (1 + h) \cdot \Big( p(r,\phi) + p_1(r,\phi,v_1) +p_2(r,\phi,v_2) + p_3(r,\phi,v_1,v_2) \Big) + v_1 \cdot \Big(\frac{\partial h_1}{\partial r} v_1 v_2 \cdot r_{0}(v_1,v_2,r,\phi)$$
$$ + \frac{\partial h_1}{\partial \phi} (\omega(r) + v_1 v_2 \cdot \omega_{0}(v_1,v_2,r,\phi))
 + \frac{\partial h_1}{\partial v_2} \cdot v_2 \cdot ( \varsigma(r,\phi) + \varsigma_1(r,\phi,v_1) + \varsigma_2(r,\phi,v_2) + \varsigma_3(r,\phi,v_1,v_2) ) \Big)$$
$$= \xi_1 \cdot \Big( p(r,\phi) + p_1(r,\phi,v_1) +p_2(r,\phi,v_2) + p_3(r,\phi,v_1,v_2) \Big) $$
$$+ v_1 \cdot \Big( \frac{\partial h_1}{\partial \phi} \omega(r) + \frac{\partial h_1}{\partial v_2} \cdot v_2 \cdot ( \varsigma(r,\phi) + \varsigma_2(r,\phi,v_2) + \varsigma_3(r,\phi,v_1,v_2) ) \Big) + h.o.t. $$

So in order to kill the term $p_2$, it suffices to find the function $h_1$ such that $h_1(r,\phi,0)=0$ and:
$$- (1+h_1) \cdot p_2 = \frac{\partial h_1}{\partial \phi} \omega(r) + \frac{\partial h_1}{\partial v_2} \cdot v_2 \cdot \Big( \varsigma(r,\phi) + \varsigma_2(r,\phi,v_2) \Big)$$
Let us take $h_1(r,\phi, v_2) = h_1(r,v_2)$, i.e. $\frac{\partial h_1}{\partial \phi} = 0$.
Consider the system:
\begin{equation}\label{}
\begin{cases}
\dot{X} = -(1+X) \cdot p_2(r,\phi,v_2) 
\\
\dot{v}_2 = v_2 \cdot (\varsigma(r,\phi) + \varsigma_2(r,\phi,v_2)) 
\end{cases}
\end{equation}
Due to $p_2(r,\phi,0) = 0$, the 
point $(X,v_2)=(0,0)$ is a fixed point.
The linearized system at $(0,0)$ is:
\begin{equation}\label{}
\begin{aligned}
\left( \begin{array}{ccc}
0 & - \frac{\partial p_2}{\partial v_2} (r,\phi,0) \\
0 & \varsigma(r,\phi) \end{array} \right) \circ \left( \begin{array}{c}
X \\
v_2  \end{array} \right) 
=
\left( \begin{array}{c}
- \frac{\partial p_2}{\partial v_2} (r,\phi,0) v_2 \\
\varsigma(r,\phi) v_2  \end{array} \right) 
\end{aligned}
\end{equation}
Note that we assume $\varsigma(r,\phi) < 0$ for all $r$ and $\phi$. As this is the only non-zero eigenvalue of the matrix governing linearized system above, then
from the strong stable manifold theorem for fixed points, follows the existence of $h_1$ which is its parameterization. 
\end{proof}

\begin{prop}
Killing the term $\varsigma_1$ with substitution $\xi_2 := v_2 + h_2(r,\phi, v_1) \cdot v_2$
\end{prop}
\begin{proof}
The only difference with respect to the previous proposition is that we need to solve the following equation for $h_2$ in order to kill the term $\varsigma_1$:
$$h_2(r,\phi,0)=0 \ \ \  \text{and} \ \ \ \ - (1+h_2) \cdot \varsigma_1 = \frac{\partial h_2}{\partial \phi} \omega(r) + \frac{\partial h_2}{\partial v_1} \cdot v_1 \cdot p(r,\phi) $$
Let $h_2(r,\phi, v_1) = h_2(r,v_1)$, i.e. $\frac{\partial h_2}{\partial \phi} = 0$.

Consider the system:
\begin{equation}\label{}
\begin{cases}
\dot{X} = -(1+X) \cdot \varsigma_1(r,\phi,v_1) 
\\
\dot{v}_1 = v_1 \cdot p(r,\phi) 
\end{cases}
\end{equation}
Due to $\varsigma_1(r,\phi,0) = 0$, the 
point $(X,v_1)=(0,0)$ is a fixed point.
The linearized system at $(0,0)$ is:

\begin{equation}\label{}
\begin{aligned}
\left( \begin{array}{ccc}
0  & -\frac{\partial \varsigma_1}{\partial v_1} (r,\phi,0) \\
0 & p(r,\phi) \end{array} \right) \circ \left( \begin{array}{c}
X \\
v_1  \end{array} \right) 
=
\left( \begin{array}{c}
- \frac{\partial \varsigma_1}{\partial v_1} (r,\phi,0) v_1 \\
p(r,\phi) v_1  \end{array} \right) 
\end{aligned}
\end{equation}
As we assume $p(r,\phi) > 0$ for all $r$ and $\phi$, 
from the strong unstable manifold theorem for fixed points, follows the existence of required $h_2$ which is its the parameterization. 
\end{proof}

\begin{prop}
Killing the term $p_1$ with substitution $\xi_1 := v_1 + h_3(r,\phi, v_1) \cdot v_1$
\end{prop}
\begin{proof}
Here in order to kill the term $p_1$, we have to look for $h_3$ such that:
$$h_3(r,\phi,0)=0 \ \ \ \text{and} \ \ \ - (1+h_3) \cdot p_1 = \frac{\partial h_3}{\partial \phi} \omega(r) + \frac{\partial h_3}{\partial v_1} \cdot v_1 \cdot \Big(p(r,\phi) + p_1(r,\phi,v_1) \Big) $$
Let $h_3(r,\phi, v_1) = h_3(r,v_1)$, i.e. $\frac{\partial h_3}{\partial \phi} = 0$.
Consider the system:
\begin{equation}\label{}
\begin{cases}
\dot{X} = -(1+X) \cdot p_1(r,\phi,v_1) 
\\
\dot{v}_1 = v_1 \cdot \Big( p(r,\phi) + p_1(r,\phi,v_1) \Big) 
\end{cases}
\end{equation}
Due to $p_1(r,\phi,0) = 0$, the 
point $(X,v_1)=(0,0)$ 
The linearized system at $(0,0)$ is:
\begin{equation}\label{}
\begin{aligned}
\left( \begin{array}{ccc}
0 &  -\frac{\partial p_1}{\partial v_1} (r,\phi,0) \\
0 & p(r,\phi) + p_1(r,\phi,0) \end{array} \right) \circ \left( \begin{array}{c}
X \\
v_1  \end{array} \right) 
=
\left( \begin{array}{c}
- \frac{\partial p_1}{\partial v_1} (r,\phi,0) v_1 \\
p(r,\phi) v_1  \end{array} \right) 
\end{aligned}
\end{equation}
As in the previous propositions, the strong unstable manifold theorem for fixed points, guarantees the existence of $h_3$. 
\end{proof}

\begin{prop}
Killing the term $\varsigma_2$ with substitution $\xi_2 := v_2 + h_4(r,\phi, v_2) \cdot v_2$
\end{prop}
\begin{proof}
In order to kill the term $\varsigma_2$, we look for $h_4$ satisfying the following conditions:
$$h_4(r,\phi,0)=0 \ \ \ \ \text{and} \ \ \ - (1+h_4) \cdot \varsigma_2 = \frac{\partial h_4}{\partial \phi} \omega(r) + \frac{\partial h_4}{\partial v_2} \cdot v_2 \cdot \Big(\varsigma(r,\phi) + \varsigma_2(r,\phi,v_2) \Big)$$
\newline
Let $h_4(r,\phi, v_2) = h_4(r,v_2)$, i.e. $\frac{\partial h_4}{\partial \phi} = 0$.
Consider the system:
\begin{equation}\label{}
\begin{cases}
\dot{X} = -(1+X) \cdot \varsigma_2(r,\phi,v_2) 
\\
\dot{v}_2 = v_2 \cdot \Big( \varsigma(r,\phi) + \varsigma_2(r,\phi,v_2) \Big) 
\end{cases}
\end{equation}
Due to $\varsigma_2(r,\phi,0) = 0$, the
point $(X,v_2)=(0,0)$ is a fixed point.
The linearized system at $(0, 0)$ is:
\begin{equation}\label{}
\begin{aligned}
\left( \begin{array}{ccc}
0 & -\frac{\partial \varsigma_2}{\partial v_2} (r,\phi,0) \\
 0 & \varsigma(r,\phi) + \varsigma_2(r,\phi,0) \end{array} \right) \circ \left( \begin{array}{c}
X \\
v_2  \end{array} \right) 
=
\left( \begin{array}{c}
- \frac{\partial \varsigma_2}{\partial v_2} (r,\phi,0) v_2 \\
\varsigma(r,\phi) v_2  \end{array} \right) 
\end{aligned}
\end{equation}
Hence from the strong stable manifold theorem for fixed point, follows the existence of $h_4$ which is its the parameterization. 
\end{proof}

In order to obtain the normal form (\ref{normFormODE}) of the ODE near the center manifold $D$, we need to utilize the lemmas above along with the following:
\begin{lemma}
There exists an analytic and convergent change of the coordinates that transforms system
\begin{equation}\label{5rown}
\begin{cases}
\dot{v_1}= p(r,\phi) \cdot v_1  + \Theta(v_1,v_2,r,\phi) \cdot v_1
\\
\dot{v_2}= \varsigma(r,\phi) \cdot v_2 +  \Lambda(v_1,v_2,r,\phi) \cdot v_2
\\
\dot{r} = v_1 v_2 \cdot r_{0v}(v_1,v_2,r,\phi) 
\\
\dot{\phi} = \omega(r) + v_1 v_2 \cdot \omega_{0v}(v_1,v_2,r,\phi) 
\end{cases}
\end{equation}
into
\begin{equation}\label{6rown}
\begin{cases}
\dot{z_1}= p_0(r) \cdot z_1  + \theta(z_1,z_2,r,\phi) \cdot z_1
\\
\dot{z_2}= \varsigma_0(r) \cdot z_2 +  \lambda(z_1,z_2,r,\phi) \cdot z_2
\\
\dot{r} = z_1 z_2 \cdot r_{0z}(z_1,z_2,r,\phi) 
\\
\dot{\phi} = \omega(r) + z_1 z_2 \cdot \omega_{0z}(z_1,z_2,r,\phi) 
\end{cases}
\end{equation}
with 
$$p_0 (r):= \int_{0}^{1}{p(r, \phi) \ d \phi}, \ \ \ \varsigma_0 (r):=  \int_{0}^{1}{\varsigma(r, \phi) \ d \phi}$$ 
and $\theta$, $\lambda$ being $1$ -periodic in $\phi$
\end{lemma}
\begin{proof}
Let us write
$$v_1:=z_1 \cdot (1+w_1(r,\phi)), \ \ \ \  v_2:=z_2 \cdot (1+w_2(r,\phi))$$
with $w_1(r,0) = w_2(r,0)=0$. Then
\begin{equation}\label{}
\begin{aligned}
\dot{v}_1 & = \dot{z}_1 \cdot \Big(1+w_1(r,\phi) \Big) + z_1 \cdot \Big(\frac{\partial w_1}{\partial r} \dot{r} + \frac{\partial w_1}{\partial \phi} \dot{\phi} \Big) =  \\
 & = \Big(p_0(r) z_1 + \theta(z_1,z_2,r,\phi) z_1 \Big) \cdot \Big(1+w_1(r,\phi)\Big)  \\
 & + z_1 \cdot \Big(\frac{\partial w_1}{\partial r}  \cdot  z_1 z_2  \cdot  r_{0}(z_1,z_2,r,\phi) (1+w_1(r,\phi))(1+w_2(r,\phi)) + \frac{\partial w_1}{\partial \phi} \omega(r) \\
 & + \frac{\partial w_1}{\partial \phi} \cdot z_1 z_2  \cdot  \omega_{0}(z_1,z_2,r,\phi) (1+w_1(r,\phi))(1+w_2(r,\phi)) \Big) 
\end{aligned}
\end{equation}
On the other hand
\begin{equation}\label{}
\begin{aligned}
\dot{v}_1 & = p(r,\phi) v_1 + \Theta(v_1,v_2,r,\phi) v_1 \\
 & = p(r,\phi) z_1 (1+w_1(r,\phi)) + \Theta(v_1,v_2,r,\phi) z_1 (1+w_1(r,\phi)) 
\end{aligned}
\end{equation}
Hence we need to solve for $w_1$ the following equation:
$$p_0(r) (1+w_1(r,\phi)) + \frac{\partial w_1}{\partial \phi} \omega(r) = p(r,\phi) (1+w_1(r,\phi))$$
which gives 
$$\frac{1}{1+w_1(r,\phi)} \cdot \frac{\partial w_1}{\partial \phi} = \frac{1}{\omega(r)} (p(r,\phi)-p_0(r))$$
and so 
$$1+w_1(r,\phi) = \exp \Big(\int_{0}^{\phi}{ \frac{p(r,\phi')-p_0(r)}{\omega(r)} d\phi'} \Big)$$
Analogously we find that 
$$1+w_2(r,\phi) = \exp \Big(\int_{0}^{\phi}{ \frac{\varsigma(r,\phi') - \varsigma_0(r)}{\omega(r)} d\phi'} \Big)$$
By comparing expressions for $\dot{v}_1$, $\dot{v}_2$ and using formulas for $w_1$, $w_2$, we get that $\theta$ and $\lambda$ are $1$ - periodic in $\phi$.
\end{proof}

Hence, the above lemma and propositions, give us the analytic and convergent normal form:
\begin{equation}\label{7rown}
\begin{cases}
\dot{v_1}= p(r) \cdot v_1  + p_0(v_1,v_2,r,\phi) \cdot {v_1}^2 v_2
\\
\dot{v_2}= \varsigma(r) \cdot v_2 +  \varsigma_0(v_1,v_2,r,\phi) \cdot v_1 {v_2}^2
\\
\dot{r} = v_1 v_2 \cdot r_{0}(v_1,v_2,r,\phi) 
\\
\dot{\phi} = \omega(r) + v_1 v_2 \cdot \omega_{0}(v_1,v_2,r,\phi) 
\end{cases}
\end{equation}
with $p_0$, $\varsigma_0$, $r_0$, $\omega_0$ being $1$ - periodic in $\phi$.

\subsection{Normal form of the return map}\label{nf2}
In this section we solve the system of equations (\ref{7rown}) near the center manifold $(v_1,v_2,r,\phi)=(0,0,r,\phi)$, find the local - Shilnikov map (proposition \ref{localMapDerivation}), and the return map to the neigbourhood of $W^c(0)$ along the non-transverse homoclinic channel (proposition \ref{fullReturnMap}).

\begin{prop}\label{localMapDerivation}
The local Shilnikov map from the cross section $S_2= \{ v_2=h \}$ to cross section $S_1= \{ v_1=h \}$ is given by :
\begin{equation}\label{trueLocal}
\begin{cases}
\bar{v}_2 = h \cdot \Big( \frac{v_1(0)}{h \cdot (1 - h v_1(0) \cdot O(1))} \Big)^{\frac{- \varsigma(r(0))}{p(r(0))}} (1 - h^2 \cdot \Big(\frac{v_1(0)}{h \cdot (1 - h v_1(0) \cdot O(1))} \Big)^{\frac{- \varsigma(r(0))}{p(r(0))}} \cdot O(1)) + h.o.t
\\
\bar{r}= r(0) + \frac{v_1(0) h}{(1 - h v_1(0) \cdot O(1))} \cdot O(1) + h.o.t.
\\
\bar{\phi}= \phi(0)+ \frac{-1}{p(r(0))} \ln \Big(\frac{v_1(0)}{h \cdot (1 - h v_1(0) \cdot O(1))} \Big) \cdot \Big(\omega \big( r(0) \big) + \omega' \big( r(0) \big) \cdot \frac{v_1(0) h \cdot O(1)}{(1 - h v_1(0) \cdot O(1))} \Big)+ \frac{v_1(0) h \cdot O(1)}{(1 - h v_1(0) \cdot O(1))} +h.o.t.
\end{cases}
\end{equation}
\end{prop}
\begin{proof}
Assume that $v_2(0)=h$ and $v_1(\tau)=h$, we want to find expression for $v_2(\tau)$ in terms of $v_1(0)$. Denote by 
$$P(s,t) := \frac{1}{-t+s} \int_{s}^{t}{p \big(r(s')\big) ds'}$$ 
and 
$$G(s) := \frac{1}{s} \int_{0}^{s}{\varsigma \big(r(s')\big) ds'}$$ 
then
$$\int_{t}^{\tau}{e^{(-\tau+s) \cdot P(s,\tau)} \dot{v}_1(s) ds} = v_1(\tau) - v_1(t) \cdot e^{(-\tau+t) \cdot P(t,\tau)} +  \int_{t}^{\tau}{p\big(r(s)\big) \cdot e^{(-\tau+s) \cdot P(s,\tau)} v_1(s) ds} = $$
$$\int_{t}^{\tau}{p\big(r(s)\big) \cdot e^{(-\tau+s) \cdot P(s,\tau)} v_1(s) ds}  + \int_{t}^{\tau}{p_0 \Big(r(s),\phi(s), v_1(s), v_2(s) \Big) \cdot v_1^2(s) v_2(s) \cdot e^{(-\tau+s) \cdot P(s,\tau)} ds}$$
after rearrangement of the terms 
$$v_1(t) = v_1(\tau) e^{(\tau-t) \cdot P(t,\tau)} - \int_{t}^{\tau}{p_0 \Big(r(s),\phi(s), v_1(s), v_2(s) \Big) \cdot v_1^2(s) v_2(s) \cdot e^{(s-t) \cdot P(t,s)} ds}$$
By solving analogously the rest of the equations, we arrive with the system:
\begin{equation}\label{IntegralSystem}
\begin{cases}
v_1(t) = v_1(\tau) e^{(\tau-t) \cdot P(t,\tau)} - \int_{t}^{\tau}{p_0 \Big(r(s),\phi(s), v_1(s), v_2(s) \Big) \cdot v_1^2(s) v_2(s) \cdot e^{(s-t) \cdot P(t,s)} ds}
\\
v_2(t) = v_2(0) e^{t \cdot G(t)} - e^{t \cdot G(t)}  \int_{0}^{t}{\varsigma_0 \Big(r(s),\phi(s), v_1(s), v_2(s) \Big) \cdot v_1(s) v_2^2(s) \cdot e^{-s \cdot G(s)} ds}
\\
r(t) = r(0)+ \int_{0}^{t}{v_1(s)v_2(s) \cdot r_0 \Big( r(s), \phi(s), v_1(s), v_2(s) \Big) ds}
\\
\phi(t) = \phi(0)+ \int_{0}^{t}{\omega \Big( r(s) \Big) ds} + \int_{0}^{t}{v_1(s)v_2(s) \cdot \omega_0 \Big( r(s), \phi(s), v_1(s), v_2(s) \Big) ds}
\end{cases}
\end{equation}
We utilize the method of the consecutive approximations of the solution with the first one taken to be:
\begin{equation}\label{firstApprox}
\begin{cases}
v_1(t) := v_1(\tau) e^{(-\tau+t) \cdot p(r(0))}
\\
v_2(t) := v_2(0) e^{t \cdot \varsigma(r(0))}
\\
r(t) := r(0)
\\
\phi(t) := \phi(0)+ \int_{0}^{t}{\omega \big( r(s) \big) ds} 
\end{cases}
\end{equation}
Plugging this into the latter system (\ref{IntegralSystem}), we obtain the second approximation of the solution:
\begin{equation}\label{}
\begin{cases}
v_1(t) = v_1(\tau) \cdot e^{(-\tau+t) \cdot p(r(0))} - v_1^2(\tau) v_2(0) \cdot e^{(-2 \tau + t ) \cdot p(r(0))}  \cdot (e^{\tau \cdot (p(r(0))+\varsigma(r(0))} - e^{t \cdot (p(r(0))+\varsigma(r(0))}) \cdot O(1)
\\
v_2(t) = v_2(0) \cdot e^{t \cdot \varsigma(r(0))} - v_1(\tau) v_2^2(0) \cdot e^{t \cdot \varsigma(r(0))}e^{-\tau \cdot p(r(0))} \cdot (e^{\tau \cdot (p(r(0))+\varsigma(r(0))} - e^{t \cdot (p(r(0))+\varsigma(r(0))}) \cdot O(1)
\\
r(t) = r(0) + v_1(\tau) v_2(0) \cdot e^{-\tau \cdot p(r(0))} \cdot (e^{t \cdot (p(r(0))+\varsigma(r(0))} - 1) \cdot O(1)
\\
\phi(t) = \phi(0)+ \int_{0}^{t}{\omega \big( r(s) \big) ds} + v_1(\tau) v_2(0) \cdot e^{-\tau \cdot p(r(0))} \cdot (e^{t \cdot (p(r(0))+\varsigma(r(0))} - 1) \cdot O(1)
\end{cases}
\end{equation}
One can prove by induction that this formulas hold for all of the successive iterations. In conclusion, the true solution of the system (\ref{IntegralSystem}) satisfies:
\begin{equation}\label{OrigSolution1}
\begin{cases}
v_1(t) = v_1(\tau) \cdot e^{(-\tau+t) \cdot p(r(0))} - v_1^2(\tau) v_2(0) \cdot e^{(-2 \tau + t ) \cdot p(r(0))} \cdot O(1)
\\
v_2(t) = v_2(0) \cdot e^{t \cdot \varsigma(r(0))} - v_1(\tau) v_2^2(0) \cdot e^{-\tau \cdot p(r(0)) + t \cdot \varsigma(r(0))} \cdot O(e^{t \cdot (p(r(0))+\varsigma(r(0))})
\\
r(t) = r(0) + v_1(\tau) v_2(0) \cdot O(e^{-\tau \cdot p(r(0))})
\\
\phi(t) = \phi(0)+ \int_{0}^{t}{\omega \big( r(s) \big) ds} + v_1(\tau) v_2(0) \cdot e^{-\tau \cdot p(r(0))} \cdot O(1)
\end{cases}
\end{equation}\newline
This means that (\ref{firstApprox}), i.e. the solution of the truncated system:
\begin{equation}\label{8rown}
\begin{cases}
\dot{v_1}= p(r) \cdot v_1 
\\
\dot{v_2}= \varsigma(r) \cdot v_2 
\\
\dot{r} = 0 
\\
\dot{\phi} = \omega(r)  
\end{cases}
\end{equation}
approximates the solution of the primary system (\ref{7rown}) at time $\tau$ with the error $v_1(\tau)v_2(0) \cdot O(e^{-\tau \cdot p(r(0))})$. Note that we assumed $p(r)>0$. The solution to the latter system (\ref{8rown}) at time $\tau$ is given by:
\begin{equation}\label{mapaLokalna}
\begin{cases}
v_1(0) = v_1(\tau) e^{-\tau \cdot p(r(0))}
\\
v_2(\tau) = v_2(0) e^{\tau \cdot \varsigma(r(0))}
\\
r(\tau) = r(0) 
\\
\phi(\tau) = \phi(0)+ \tau \cdot \omega \big( r(0) \big)
\end{cases}
\end{equation}
We find the solution to (\ref{OrigSolution1}) in the following way. 
The first equation gives:
$$v_1(t) = v_1(\tau) e^{(-\tau+t) \cdot p(r(0))} (1 - v_1(\tau) v_2(0) \cdot e^{-\tau \cdot p(r(0))} \cdot O(1))$$
from which we find
$$v_1(0) = v_1(\tau) e^{-\tau \cdot p(r(0))} (1 - v_1(\tau) v_2(0) \cdot e^{-\tau \cdot p(r(0))} \cdot O(1))$$
and so
$$e^{-\tau \cdot p(r(0))} = \frac{v_1(0)}{v_1(\tau) \cdot (1 - v_1(\tau) v_2(0) \cdot e^{-\tau \cdot p(r(0))} \cdot O(1))}$$
which can be substituted once again in the denominator. Taking into account that $\tau \rightarrow +\infty$ as $v_1(0) \rightarrow 0$:
$$e^{-\tau \cdot p(r(0))} = \frac{v_1(0)}{v_1(\tau) \cdot (1 - v_1(\tau) v_2(0) \cdot \frac{v_1(0)}{v_1(\tau) \cdot (1 - v_1(\tau) v_2(0) e^{-\tau \cdot p(r(0))} \cdot O(1))} \cdot O(1))}$$
In consequence one can rewrite the latter equation up to higher order terms:
$$e^{-\tau \cdot p(r(0))} = \frac{v_1(0)}{v_1(\tau) \cdot (1 - v_2(0) v_1(0) \cdot O(1))}$$
From which we can read the approximated formula for $\tau$:
$$\tau = \frac{-1}{p(r(0))} \cdot \ln \Big(\frac{v_1(0)}{v_1(\tau) \cdot (1 - v_2(0) v_1(0) \cdot O(1))} \Big)$$
which we plug into the equations for $v_2(\tau)$, $r(\tau)$, $\phi(\tau)$
$$v_2(\tau) = v_2(0) \cdot e^{\tau \cdot \varsigma(r(0))}  \cdot (1 - v_1(\tau) v_2(0) \cdot e^{\tau \cdot \varsigma(r(0))} \cdot O(1))$$
and obtain
\begin{equation}\label{}
\begin{cases}
v_2(\tau) = v_2(0) \cdot (\frac{v_1(0)}{v_1(\tau) \cdot (1 - v_2(0) v_1(0) \cdot O(1))})^{\frac{- \varsigma(r(0))}{p(r(0))}} \cdot (1 - v_1(\tau) v_2(0) \cdot (\frac{v_1(0)}{v_1(\tau) \cdot (1 - v_2(0) v_1(0) \cdot O(1))})^{\frac{- \varsigma(r(0))}{p(r(0))}} O(1)) + h.o.t 
\\
r(\tau) = r(0) + \frac{v_1(0) v_2(0)}{(1 - v_2(0) v_1(0) \cdot O(1))} \cdot O(1) 
\\
\phi(\tau) = \phi(0)+ \int_{0}^{\tau}{\omega \big( r(s) \big) ds} + \frac{v_1(0) v_2(0)}{(1 - v_2(0) v_1(0) \cdot O(1))} \cdot O(1)+ h.o.t.
\end{cases}
\end{equation}
After expanding $\int_{0}^{\tau}{\omega \big( r(s) \big) ds}$ with the use of approximation $r(t) \approx r(0) + v_1(\tau) v_2(0) \cdot O(e^{-\tau \cdot p(r(0))})$ (taken from the equation (\ref{OrigSolution1})), we find that
$$\phi(\tau) = \phi(0)+ \tau \cdot (\omega \big( r(0) \big) + \omega' \big( r(0) \big) \cdot \frac{v_1(0) v_2(0)}{(1 - v_2(0) v_1(0) \cdot O(1))} \cdot O(1) + h.o.t.)+ \frac{v_1(0) v_2(0)}{(1 - v_2(0) v_1(0) \cdot O(1))} \cdot O(1)$$
which finally gives us
\begin{equation}\label{trueLocal}
\begin{cases}
v_2(\tau) = v_2(0) \cdot \Big( \frac{v_1(0)}{v_1(\tau) \cdot (1 - v_2(0) v_1(0) \cdot O(1))} \Big)^{\frac{- \varsigma(r(0))}{p(r(0))}} (1 - v_1(\tau) v_2(0) \cdot \Big(\frac{v_1(0)}{v_1(\tau) \cdot (1 - v_2(0) v_1(0) \cdot O(1))} \Big)^{\frac{- \varsigma(r(0))}{p(r(0))}} \cdot O(1)) + h.o.t
\\
r(\tau) = r(0) + \frac{v_1(0) v_2(0)}{(1 - v_2(0) v_1(0) \cdot O(1))} \cdot O(1) + h.o.t.
\\
\phi(\tau) = \phi(0)+ \frac{-1}{p(r(0))} \ln \Big(\frac{v_1(0)}{v_1(\tau) \cdot (1 - v_2(0) v_1(0) \cdot O(1))} \Big) \cdot \Big(\omega \big( r(0) \big) + \frac{\omega' g( r(0)) v_1(0) v_2(0) \cdot O(1)}{(1 - v_2(0) v_1(0) \cdot O(1))} + h.o.t. \Big)+  \frac{v_1(0) v_2(0) \cdot O(1)}{(1 - v_2(0) v_1(0) \cdot O(1))}
\end{cases}
\end{equation}
\end{proof}

\begin{rmq}
Note that all of the performed transformations were analytic. The local Shilnikov map obviously is analytic as well.
\end{rmq}
\begin{ass}
We assume that the global map from the cross section $S_1 = \{ v_1=h \}$ to $S_2 = \{ v_2=h \}$ along the non-transverse homoclinic channel $I$ is analytic.
\end{ass}
\begin{prop}
The truncated global map "$glob$" from the cross section $S_1 = \{ v_1=h \}$ to $S_2 = \{ v_2=h \}$ along the non-transverse homoclinic channel is given by:
\begin{equation}
\left(
\begin{array}{c}
v_2\\
r\\
\phi
\end{array}
\right)
\mapsto
\left(
\begin{array}{c}
a_1(r,\phi) \cdot v_2 \\
b_0(r,\phi) + b_1(r,\phi) \cdot v_2\\
c_0(r,\phi) + c_1(r,\phi)\cdot v_2)
\end{array}
\right)
\end{equation}

\end{prop}
\begin{proof}
Let us expand the map 
$$glob : S_1 \ni (v_2,r,\phi) \mapsto  \big( glob_1(v_2,r,\phi), \ glob_2(v_2,r,\phi), \ glob_3(v_2,r,\phi) \big) \in S_2$$
in terms of $v_2$ - being very small, that is:
\begin{equation}
glob(v_2,r,\phi) = 
\left(
\begin{array}{c}
a_0(r,\phi) + a_1(r,\phi) \cdot v_2 + h.o.t. \\
b_0(r,\phi) + b_1(r,\phi) \cdot v_2 + h.o.t. \\
c_0(r,\phi) + c_1(r,\phi) \cdot v_2 + h.o.t.
\end{array}
\right)
\end{equation}

Due to the existence of the homoclinic channel $I$ we have $a_0=0$. 
\end{proof}

In the following theorem we derive the formula for the return map along the homoclinic channel $I$: 

\begin{thm}\label{fullReturnMap}
The return map $loc \circ glob : \ S_1  \rightarrow \ S_1$, in suitable coordinates $(z,r,\phi)$, with $z := ln(\frac{v_2}{h})$ 
is given by:
\newline
\begin{equation}
\left(
\begin{array}{c}
z\\
r\\
\phi
\end{array}
\right)
\mapsto
\left(
\begin{array}{c}
\Big( \ln(a_1(r,\phi)) + z \Big) \cdot \frac{- \varsigma(b_0(r,\phi))}{p(b_0(r,\phi))} + O(z \cdot e^z) \\b_0(r, \phi) + O(e^z) \\
(c_0(r,\phi) - \frac{\omega ( b_0(r,\phi))}{p(b_0(r,\phi))} \cdot \ln(a_1(r,\phi))) + \frac{- \omega ( b_0(r,\phi))}{p(b_0(r,\phi))} \cdot z + O(z \cdot e^z) \ \ (mod 1)
\end{array}
\right)
\end{equation}
\end{thm}
\begin{proof}
In order to compose the local map with the global return map between the sections $S_1$ and $S_2$ we need to perform the following substitutions 
\begin{equation}
\begin{cases}
v_1(0)= a_1(r,\phi) \cdot v_2 + h.o.t.
\\
r(0) = b_0(r,\phi) + b_1(r,\phi) \cdot v_2 +h.o.t.
\\
\phi(0) = c_0(r,\phi) + c_1(r,\phi)\cdot v_2 +h.o.t.
\end{cases}
\end{equation}
in the equations (\ref{trueLocal}), which brings us the following form of the return map:
\begin{equation}\label{return1}
\begin{cases}
\bar{v}_2 = h \Big (\frac{ a_1(r,\phi) \cdot v_2 + h.o.t.}{h \cdot (1 - h ( a_1(r,\phi) \cdot v_2) \cdot O(1))} \Big)^{\frac{- \varsigma(b_0(r,\phi) + b_1(r,\phi) \cdot v_2)}{p(b_0(r,\phi) + b_1(r,\phi) \cdot v_2)}} \cdot \Big(1 - h^2 (\frac{ a_1(r,\phi) \cdot v_2}{h \cdot (1 - h ( a_1(r,\phi) \cdot v_2) \cdot O(1))})^{\frac{- \varsigma(b_0(r,\phi) + b_1(r,\phi) \cdot v_2)}{p(b_0(r,\phi) + b_1(r,\phi) \cdot v_2)}} \cdot O(1)\Big) + h.o.t
\\
\bar{r} = b_0(r,\phi) + b_1(r,\phi) \cdot v_2 + \frac{h \cdot ( a_1(r,\phi) \cdot v_2)}{(1 - h ( a_1(r,\phi) \cdot v_2) \cdot O(1))} \cdot O(1) + h.o.t.
\\
\bar{\phi} =  c_0(r,\phi) + c_1(r,\phi)\cdot v_2 + \frac{-1}{p(b_0(r,\phi) + b_1(r,\phi) \cdot v_2)} \cdot \ln \Big(\frac{ a_1(r,\phi) \cdot v_2}{h \cdot (1 - h ( a_1(r,\phi) \cdot v_2) \cdot O(1))}\Big) \cdot \omega \big( b_0(r,\phi) + b_1(r,\phi) \cdot v_2 \big) + 
\\ 
\ \ \ \ \ \ \ \ \ \Bigg(\frac{- \ln \big(\frac{ a_1(r,\phi) \cdot v_2}{h \cdot (1 - h ( a_1(r,\phi) \cdot v_2) \cdot O(1))}\big)}{p(b_0(r,\phi) + b_1(r,\phi) \cdot v_2)} \cdot \omega' \Big( b_0(r,\phi) + b_1(r,\phi) \cdot v_2 \Big) +1 \Bigg) \cdot \frac{h ( a_1(r,\phi) \cdot v_2)}{(1 - h ( a_1(r,\phi) \cdot v_2) \cdot O(1))} \cdot O(1) +... + h.o.t.
\end{cases}
\end{equation}
which can be further simplified to
\begin{equation}\label{return1}
\begin{cases}
\bar{v}_2 = h \Big (\frac{ a_1(r,\phi) \cdot v_2 + h.o.t.}{h \cdot (1 - h ( a_1(r,\phi) \cdot v_2) \cdot O(1))} \Big)^{\frac{- \varsigma(b_0(r,\phi) + b_1(r,\phi) \cdot v_2)}{p(b_0(r,\phi) + b_1(r,\phi) \cdot v_2)}}  + h.o.t
\\
\bar{r} = b_0(r,\phi) + b_1(r,\phi) \cdot v_2 + \frac{h \cdot ( a_1(r,\phi) \cdot v_2)}{(1 - h ( a_1(r,\phi) \cdot v_2) \cdot O(1))} \cdot O(1) + h.o.t.
\\
\bar{\phi} =  c_0(r,\phi) + c_1(r,\phi)\cdot v_2 + \frac{-1}{p(b_0(r,\phi) + b_1(r,\phi) \cdot v_2)} \cdot \ln \Big(\frac{ a_1(r,\phi) \cdot v_2}{h \cdot (1 - h ( a_1(r,\phi) \cdot v_2) \cdot O(1))}\Big) \cdot \omega \big( b_0(r,\phi) \big) + h.o.t.
\end{cases}
\end{equation}

the substitution
$$z := ln \Big(\frac{v_2}{h} \Big), \ \ \ \text{so: \ \ \ } v_2= h \cdot e^z \ \ \ \text{and \ \ \ } z \approx - \infty \ \ \text{for \ \ \ } v_2 \approx 0 $$
performed in the equation (\ref{return1}), together with simplifications, leads us to:
\begin{equation}\label{v2rphiReturn}
\begin{cases}
\bar{v}_2 = h \cdot  \Big(a_1(r,\phi) \cdot e^z  \Big)^{\cdot \frac{- \varsigma(b_0(r,\phi))}{p(b_0(r,\phi))} + O(z \cdot e^z)} + h.o.t
\\
\bar{r} = b_0(r,\phi) + b_1(r,\phi) \cdot O(e^z) + h e^{z} \cdot O(1) + h.o.t.
\\ 
\bar{\phi} = c_0(r,\phi) + c_1(r,\phi) \cdot h e^z - \frac{\omega(b_0(r,\phi))}{p(b_0(r,\phi))} \cdot  \Big(z + \ln(a_1(r,\phi))  \Big) + O(z \cdot e^z) + h.o.t.
\end{cases}
\end{equation}
In conclusion:
\begin{equation}\label{}
\begin{cases}
\bar{z} =  \Big( \ln(a_1(r,\phi)) + z  \Big) \cdot \frac{- \varsigma(b_0(r,\phi))}{p(b_0(r,\phi))} + O(z \cdot e^z) 
\\
\bar{r} = b_0(r,\phi) + O(e^z)
\\ 
\bar{\phi} =  \Big(c_0(r,\phi) - \frac{\omega ( b_0(r,\phi))}{p(b_0(r,\phi))} \cdot \ln(a_1(r,\phi))  \Big)  + \frac{- \omega ( b_0(r,\phi))}{p(b_0(r,\phi))} \cdot z + O(z \cdot e^z) \ \ \ (\ mod \ 1 \ )
\end{cases}
\end{equation}
\end{proof}
We are able to write the return map along the homoclinic channel $I$ in a general and concise way due to the following
\begin{prop}\label{rescaling}
After rescaling variable $z$, the return map takes the form:
\newline
\begin{equation}\label{rescaledReturnformula}
\left(
\begin{array}{c}
z\\
r\\
\phi
\end{array}
\right)
\mapsto
\left(
\begin{array}{c}
\Omega(r,\phi) + \Gamma(r,\phi) \cdot z + O(z \cdot e^z)\\
b_0(r, \phi) + O(e^z)\\
c(r,\phi) + z + O(z \cdot e^z)  \ \ (mod \ 1)
\end{array}
\right)
\end{equation}
with
$$\Omega(r,\phi):= \frac{\omega ( b_0(r,\phi))}{p(b_0(r,\phi))} \cdot \ln(a_1(r,\phi)) \cdot \frac{- \varsigma(b_0(r,\phi))}{p(b_0(r,\phi))} 
$$ 
$$\Gamma(r,\phi):=\frac{- \varsigma(b_0(r,\phi))}{p(b_0(r,\phi))} $$
$$c(r, \phi) := c_0(r,\phi) - \frac{\omega ( b_0(r,\phi))}{p(b_0(r,\phi))} \cdot \ln(a_1(r,\phi)) \Big)$$
\end{prop}
\begin{proof}
If we now denote $\alpha(r):=-\frac{\omega(r)}{p(r)}$ and $H (z,r,\phi):= \big(z \cdot \alpha(r), \ r, \ \phi \big)$ then
\newline
$$H \circ loc \circ glob \circ H^{-1} (z,r,\phi) =  \Bigg( \ \Big( \ln(a_1(r,\phi)) \cdot \alpha(b_0(r,\phi))  + z \Big) \cdot \frac{- \varsigma(b_0(r,\phi))}{p(b_0(r,\phi))} + O(z \cdot e^z), $$
$$\  b_0(r, \phi) + O(e^z), \ \  \Big(c_0(r,\phi) - \frac{\omega ( b_0(r,\phi))}{p(b_0(r,\phi))} \cdot \ln(a_1(r,\phi)) \Big) + z + O(z \cdot e^z) \ \ (mod \ 1) \ \Bigg) =$$
$$\Big(\Omega(r,\phi) + \Gamma(r,\phi) \cdot z + O(z \cdot e^z), \  b_0(r, \phi) + O(e^z), \ c(r,\phi) + z + O(z \cdot e^z)  \ \ (mod \ 1) \ \Big)$$
\end{proof}
In the following two sections, we will investigate more carefully its truncated version.
\newline
\begin{equation}\label{truncatedReturnformula}
\left(
\begin{array}{c}
z\\
r\\
\phi
\end{array}
\right)
\mapsto
\left(
\begin{array}{c}
\Omega(r,\phi) + \Gamma(r,\phi) \cdot z\\
b_0(r, \phi)\\
c(r,\phi) + z\ \ (mod \ 1)
\end{array}
\right)
\end{equation}

\begin{prop}
The original return map satisfies condition:
$$Ret_{S_1} (z,r,\phi) = loc \circ glob (z,r,\phi) +  O(z \cdot e^{-|z|})$$
\end{prop}

\begin{rmq}
All of the performed transformations were analytic, so the return map in the $(z,r,\phi)$ coordinates is analytic.
\end{rmq}


\subsection{The limit dynamics of the return map}\label{sec:foliation}
In this section we will prove that in the special case $-\frac{\varsigma(b_0(r,\phi))}{p(b_0(r,\phi))} = \Gamma(r,\phi) = \Gamma \in \mathbb N$, $\Gamma >1$, the dynamics of the return map are in the limit $v_2 \to 0$ the same as the dynamics of its truncated version. This follows immediately from the theorem \ref{foliation}. Before stating it we need to introduce the objects which we will use in this section. 
\newline
Let us consider the following mapping $\varpi$:
\begin{equation}\label{varpi}
\begin{aligned}
\varpi: \ \mathbb R \times \mathbb R \times \mathbb R \times S^1 \ni \left( \begin{array}{c}
v_2 \\
z \\
r \\ 
\phi \\ 
\end{array} \right) 
\mapsto
\left( \begin{array}{c}
h \cdot \big(\frac{a_1(r,\phi) \cdot v_2}{h} \big)^{\frac{-\varsigma(b_0(r,\phi))}{p(b_0(r,\phi))}} + h.o.t.(v_2) \\
\Omega(r,\phi) + \Gamma(r,\phi) \cdot z + O( v_2 \cdot \ln(v_2)) \ (mod \ 1) \\
b_0(r,\phi) + O( v_2  ) \\
c(r,\phi) + z + O( v_2 \cdot \ln(v_2) ) \ \ (mod \ 1 \ )  \\ 
\end{array} \right) \in \mathbb R \times \mathbb R \times \mathbb R \times S^1 
\end{aligned}
\end{equation}
\begin{rmq}
In the definition of the mapping $\varpi$, the variables $v_2$ and $z$ are independent.
\newline
Observe that: $\varpi(0,z,r,\phi) = \big(0, \ loc \circ glob (z,r,\phi) \big)$. 
\newline
On the other hand, $\varpi(v_2,\ln(\frac{v_2}{h}),r,\phi) = \big(h \cdot \big(\frac{a_1(r,\phi) \cdot v_2}{h} \big)^{\frac{-\varsigma(b_0(r,\phi))}{p(b_0(r,\phi))}} + h.o.t.(v_2), Ret_{S_1}(\ln(\frac{v_2}{h}),r,\phi) \big)$. 
\end{rmq}
\begin{df}
We say that a foliation 
$\mathcal{L} = \{h(y) \}$ with $C^m$ leaves  ($m \geq 0$) given by the graphs of functions $x=h(y)$ is given in a domain $W=\{ (x,y) \ | \ y \geq 0, \ x \in S^1 \times [0,1] \times S^1 \}$ if the following conditions are satisfied:
\newline
1. the domain of any function $h(y) \in \mathcal{L}$ is an open connected set in $\mathbb R$ and its graph lies entirely in $W$
\newline
2. for any point $(x^0,y^0) \in W$ there exists a unique function $h(y) \in \mathcal{L}$ such that the domain of $h(y)$ contains $x^0$ and $h(y^0)=x^0$ ( we denote this function by $h(y; x^0,y^0)$
\newline
3. the function $h(y;x^0,y^0)$ is a $C^m$ function of $y$ for fixed $x^0$ and $y^0$
\newline
The graphs of the functions $h(y)$ are called the leaves of $\mathcal{L}$ and are denoted by the same letter $h$.
\end{df}
\begin{df}
A foliation $\mathcal{L}$ is said to be $C^l$-smooth ($l \geq 0$) if $h(y;x^0,y^0)$ is a $C^l$ function of $(y;x^0,y^0)$, i.e., all partial derivatives of order $\leq l$ are continuous and uniformly bounded. 
\end{df}
\begin{df}
A foliation $\mathcal{L}$ of $W$ is said to be $\varpi$- invariant  if 
for any leaf $h \in \mathcal{L}$, $h \neq W_0$, there exists a leaf $\bar{h} \in \mathcal{L}$ such that $\varpi(h) \subseteq \bar{h}$
\end{df}
\begin{df}
An arbitrary vector function $\mu(x^0,y^0)$ on $W$ is called a field of hyperplanes. If $\mu(x^0,y^0) = \frac{\partial h}{\partial y}(y;x^0,y^0)|_{y=y^0}$ for some $h \in \mathcal{L}$, then the field $\mu(x^0,y^0)$ is called a field of tangent hyperplanes to $\mathcal{L}$. 
\end{df}
\begin{thm}\label{foliation}
Let us rewrite 
$$x:=(z,r,\phi), \ \ y:=v_2, \ \ (F,G):=\varpi$$ 
so that:
\begin{equation}\label{ShShFG}
\begin{cases}
\bar{x} =  F(x,y)
\\
\bar{y} = G(x,y)
\end{cases}
\end{equation}
Assume that $|det(\frac{\partial F}{\partial x} (x,y))|$ is bounded from below by a positive constant independent of $x$ and $y$.
\newline
Then there exists a $\varpi$ - invariant $C^1$ foliation $\mathcal{L} = \{h(y)\}$ (satisfying the above definitions) of the space $W=\{ (x,y) \ | \ y \geq 0, \ x \in S^1 \times [0,1] \times S^1 \}$, with $C^2$ leaves given by the graphs of the functions $x=h(y)$ (hence $(z,r,\phi) = h(v_2)$).
\end{thm}
\begin{proof}
We utilize the method described inter alia in \cite{Shashkov}, i.e. we find the field of tangent hyperplanes to $\mathcal{L}$. 
Let us denote the following derivatives
\begin{equation}\label{}
\begin{cases}
A(x,y) :=  \frac{\partial F}{\partial x} (x,y) 
\\
B(x,y) :=  \frac{\partial F}{\partial y} (x,y) 
\\
C(x,y) :=  \frac{\partial G}{\partial x} (x,y) 
\\
D(x,y) :=  \frac{\partial G}{\partial y} (x,y) 
\\
\end{cases}
\end{equation}
All of the entries of $A$ are continuous and bounded as a functions $r$ and $\phi$, moreover by assumption $|det(A)|$ is bounded from below by a positive constant independent of $r,\phi$. Then one can easily find constants $A_3, B_2, C_2, D_2 >0$ such that the following holds for any $x$ and for any $y$ sufficiently small:
\begin{equation}\label{}
\begin{aligned}
&  ||A^{-1}(x,y)|| \leq A_3 
\\ 
&  ||B(x,y)|| \leq B_2 \cdot |\ln y|
\\
&  ||C(x,y)|| \leq C_2 \cdot |y|^{-\frac{\varsigma(b_0(r,\phi))}{p(b_0(r,\phi))}} \cdot |\ln y|
\\
&  ||D(x,y)|| \leq D_2 \cdot |y|^{-\frac{\varsigma(b_0(r,\phi))+p(b_0(r,\phi))}{p(b_0(r,\phi))}}  
\end{aligned}
\end{equation}
Notice that for any invertible square matrix $P$, we have $||P^{-1}||=(\min_{||w||=1} ||P w||)^{-1}$.
\newline
For $\mu$ - the field of hyperplanes to be the field of tangent hyperplanes to an invariant foliation of the form $\{x=h(y)\}$, the following condition has to be satisfied:
\begin{equation}\label{}
\begin{cases}
\frac{d x}{d y} =  \mu(x,y)
\\
\frac{d \bar{x}}{d \bar{y}} =  \mu(\bar{x}, \bar{y})
\end{cases}
\end{equation}
from (\ref{ShShFG}) we obtain after differentiation 
\begin{equation}\label{}
\begin{cases}
d \bar{x} = \frac{\partial F}{\partial x} dx + \frac{\partial F}{\partial y} dy = (A(x,y) \cdot \mu(x,y) + B(x,y)) dy
\\
d \bar{y} = \frac{\partial G}{\partial x} dx + \frac{\partial G}{\partial y} dy = (C(x,y) \cdot \mu(x,y) + D(x,y)) dy
\end{cases}
\end{equation}
which leads us to 
$$ (A(x,y) \cdot \mu(x,y) + B(x,y)) = \mu(\bar{x}, \bar{y}) \cdot (C \cdot \mu(x,y) + D(x,y))$$
and in consequence
$$ \mu(x,y) = (A(x,y) - \mu(\bar{x}, \bar{y}) \cdot C(x,y))^{-1} \cdot (\mu(\bar{x}, \bar{y}) \cdot D(x,y) - B(x,y))$$
Where $(\bar{x},\bar{y}) = (F,G)(x,y)$ and $V$ is the space of all fields of hyperplanes $\mu$ such that
\newline
1) $\mu(x,y)$ is a continuous vector function in $W$ 
\newline
2) $\sup _{\{x \in S^1 \times [0,1] \times S^1, \ y>0 \}} || \frac{\mu(x,y)}{\ln y} ||  \leq E$ for some $E>0$
\newline
3) $\sup _{\{x \in S^1 \times [0,1] \times S^1 \}} || \frac{\mu(x,y)}{\ln y} || =0 $ as $y \rightarrow 0$
\\ \\
Space $V$ equipped with a metric 
$$d(\mu_1,\mu_2) := \sup _{\{x \in S^1 \times [0,1] \times S^1, \ y>0 \}} \big|\big| \frac{\mu_1(x,y)-\mu_2(x,y)}{\ln y} \big|\big|$$
is a complete metric space. Consider the mapping $\Gamma_V : V \rightarrow V$ defined as follows:
\begin{equation}\label{}
\Gamma_V(\mu)(x,y) := (A(x,y) - \mu(\bar{x}, \bar{y}) \cdot C(x,y))^{-1} \cdot (\mu(\bar{x}, \bar{y}) \cdot D(x,y) - B(x,y)) 
\end{equation}
The fixed points of $\Gamma_V$ correspond to the field of tangent hyperplanes to an invariant foliation. Moreover, as shown in \cite{Shashkov} (lemma 2 in \cite{Shashkov}), if $\mu_1, \ \mu_2$ are the fields of tangent hyperplanes corresponding to invariant foliations $\mathcal{L}_1, \ \mathcal{L}_2$, $\mu_1 = \Gamma_V(\mu_2)$ and $h_1 \in \mathcal{L}_1$, then there exist $h_2 \in \mathcal{L}_2$ such that $(F,G)(h_1) \subset h_2$.
Let 
$$M(y):=\frac{||A^{-1}||^{-1} - ||D|| - \sqrt{(||A||^{-1} - ||D||)^2 - 4 \cdot ||B|| \cdot ||C||}}{2 \cdot ||C||}
$$ 
Note that $M(y) < \frac{||A^{-1}||^{-1}}{2 \cdot ||C||}$ and for $|y|$ sufficiently small:
$$M(y)=\frac{2 \cdot ||B||}{||A^{-1}||^{-1} - ||D|| + \sqrt{(||A^{-1}||^{-1} - ||D||)^2 - 4 \cdot ||B|| \cdot ||C||}} \leq 4 \cdot A_3 \cdot B_2 |\ln y|
$$ 
Hence the ball $K_M:= \{ \mu \in V \ | \ d(\mu,0) \leq d(0, M(y)) \}$
is $\Gamma_V$ invariant and is a complete metric space with the metric $d$. The mapping $\Gamma_V|_{K_M}$ is continuous and also is a contraction since:
$$ || \Gamma_V(\mu_1) - \Gamma_V(\mu_2) || = ||(A - \mu_1 \cdot C)^{-1}(\mu_1 \cdot D - B) - (A - \mu_2 \cdot C)^{-1}(\mu_2 \cdot D - B) || = 
$$
$$||(A - \mu_1 \cdot C)^{-1}(\mu_1 \cdot D - B) - (A - \mu_1 \cdot C)^{-1}(\mu_2 \cdot D - B) + (A - \mu_1 \cdot C)^{-1}(\mu_2 \cdot D - B) - (A - \mu_2 \cdot C)^{-1}(\mu_2 \cdot D - B) || = 
$$
$$||(A - \mu_1 \cdot C)^{-1}(\mu_1 -\mu_2) \cdot D + ((A - \mu_1 \cdot C)^{-1}- (A - \mu_2 \cdot C)^{-1})(\mu_2 \cdot D - B) || \leq 
$$
$$||(A - \mu_1 \cdot C)^{-1}|| \cdot ||(\mu_1 -\mu_2)|| \cdot ||D|| + ||(A - \mu_1 \cdot C)^{-1}\cdot(\mu_1 - \mu_2) \cdot C \cdot (A - \mu_2 \cdot C)^{-1} \cdot (\mu_2 \cdot D - B)|| \leq 
$$
$$||(A - \mu_1 \cdot C)^{-1}|| \cdot ||(\mu_1 -\mu_2)|| \cdot ||D|| + ||(A - \mu_1 \cdot C)^{-1}|| \cdot ||\mu_1 - \mu_2|| \cdot ||C|| \cdot ||(A - \mu_2 \cdot C)^{-1}|| \cdot ||(\mu_2 \cdot D - B)|| =
$$
$$||(\mu_1 -\mu_2)|| \cdot \big(||(A - \mu_1 \cdot C)^{-1}||  \cdot ||D|| + ||(A - \mu_1 \cdot C)^{-1}|| \cdot ||C|| \cdot ||(A - \mu_2 \cdot C)^{-1}|| \cdot ||(\mu_2 \cdot D - B)|| \big)
$$
$$||(\mu_1 -\mu_2)|| \cdot \big(||(A - \mu_1 \cdot C)^{-1}||  \cdot ||D|| + ||(A - \mu_1 \cdot C)^{-1}|| \cdot ||C|| \cdot ||(A - \mu_2 \cdot C)^{-1}|| \cdot (||\mu_2|| \cdot ||D|| + ||B||) \big)
$$
We can estimate the coefficient $||(A - \mu_1 \cdot C)^{-1}||$ as follows
$$||(A - \mu_1 \cdot C)^{-1}|| = (\min_{||w||=1} ||A \cdot w- \mu_1 \cdot C \cdot w||)^{-1} \leq 
(\min_{||w||=1} (||A \cdot w||- ||\mu_1 \cdot C \cdot w||))^{-1}
$$
$$\leq (\min_{||w||=1} ||A \cdot w|| + \min_{||w||=1}(- ||\mu_1 \cdot C \cdot w||))^{-1}
 \leq (\min_{||w||=1} ||A \cdot w|| - \max_{||w||=1}(||\mu_1 \cdot C \cdot w||))^{-1}
$$
$$ \leq (||A^{-1}||^{-1} - \max_{||w||=1}(||\mu_1 \cdot C \cdot w||))^{-1}
 \leq (||A^{-1}||^{-1} - ||\mu_1|| \cdot ||C||)^{-1}
$$
$$ \leq (A_3^{-1} - M \cdot C_2 \cdot |y|^{-\frac{\varsigma(b_0(r,\phi))}{p(b_0(r,\phi))}} \cdot |\ln y|)^{-1} 
$$
Hence, we can provide an estimate on the Lipshitz constant of : 
$$||(A - \mu_1 \cdot C)^{-1}||  \cdot ||D|| + ||(A - \mu_1 \cdot C)^{-1}|| \cdot ||C|| \cdot ||(A - \mu_2 \cdot C)^{-1}|| \cdot (||\mu_2|| \cdot ||D|| + ||B||)
$$
$$\leq  \frac{(D_2  |y|^{-\frac{\varsigma(b_0(r,\phi))+p(b_0(r,\phi))}{p(b_0(r,\phi))}} + (A_3^{-1} - M C_2 |y|^{-\frac{\varsigma(b_0(r,\phi))}{p(b_0(r,\phi))}} \cdot |\ln y|)^{-1}  C_2  |y|^{-\frac{\varsigma(b_0(r,\phi))}{p(b_0(r,\phi))}}  |\ln y| \cdot (M  D_2 |y|^{-\frac{\varsigma(b_0(r,\phi))+p(b_0(r,\phi))}{p(b_0(r,\phi))}} + B_2 |\ln y| )}{A_3^{-1} - M \cdot C_2 \cdot |y|^{-\frac{\varsigma(b_0(r,\phi))}{p(b_0(r,\phi))}} \cdot |\ln y|}
$$
so it is smaller than $1$ for $y$ - small enough and all $z,r,\phi$. Hence, $\Gamma_V$ possesses a unique fixed point $\mu_0$ in $K_M$, so  
\begin{equation}\label{mu0}
\mu_0(x,y)=(A(x,y) - \mu_0(\bar{x}, \bar{y}) \cdot C(x,y))^{-1} \cdot (\mu_0(\bar{x}, \bar{y}) \cdot D(x,y) - B(x,y))
\end{equation}
In consequence $||\mu_0|| = O(\ln y)$ and moreover  $||\frac{\partial \mu_0 (x,y)}{\partial (x,y)}|| = O \big(\frac{1}{y} \big)$. 
\newline
Furthermore, formal differentiation of the equation $\mu=\Gamma_V(\bar{\mu})$ with respect to $(x,y)$ gives the formula for 
$$\frac{\partial \mu (x,y)}{\partial (x,y)} = \Big(A(x,y) - \bar{\mu}(\bar{x},\bar{y}) \cdot C(x,y) \Big)^{-1} \cdot \Big(\frac{\partial \bar{\mu}(\bar{x},\bar{y})}{\partial (\bar{x},\bar{y})} \cdot \frac{\partial (F,G)(x,y)}{\partial (x,y)} \cdot D(x,y) + \bar{\mu}(\bar{x},\bar{y}) \cdot \frac{\partial D(x,y)}{\partial (x,y)} - \frac{B(x,y)}{\partial (x,y)} \Big) \ + $$
$$\Big( A(x,y) - \bar{\mu}(\bar{x},\bar{y}) \cdot C(x,y) \Big)^{-1} \cdot \Big(\frac{\partial A(x,y)}{\partial (x,y)} - \frac{\partial \bar{\mu}(\bar{x},\bar{y})}{\partial (\bar{x},\bar{y})} \cdot \frac{\partial (F,G)(x,y)}{\partial (x,y)} \cdot C(x,y) \ - 
$$
$$
- \bar{\mu}(\bar{x},\bar{y}) \cdot \frac{\partial C(x,y)}{\partial (x,y)} \Big) \cdot \Big(A(x,y) - \bar{\mu}(\bar{x},\bar{y}) \cdot C(x,y) \Big)^{-1} \cdot (\bar{\mu}(\bar{x}, \bar{y}) \cdot D(x,y) - B(x,y)) $$ 
So the equation above suggests to consider, for given $\bar{\mu}$, the following mapping $P^{\bar{\mu}} : N \rightarrow N$
$$ P^{\bar{\mu}}(\bar{\eta})(x,y) := \Big(A(x,y) - \bar{\mu}(\bar{x},\bar{y}) \cdot C(x,y) \Big)^{-1} \cdot \Big(\bar{\eta}(\bar{x},\bar{y}) \cdot \frac{\partial (F,G)(x,y)}{\partial (x,y)} \cdot D(x,y) + \bar{\mu}(\bar{x},\bar{y}) \cdot \frac{\partial D(x,y)}{\partial (x,y)} - \frac{B(x,y)}{\partial (x,y)} \Big) \ +$$
$$\Big( A(x,y) - \bar{\mu}(\bar{x},\bar{y}) \cdot C(x,y) \Big)^{-1} \cdot \Big(\frac{\partial A(x,y)}{\partial (x,y)} - \bar{\eta}(\bar{x},\bar{y}) \cdot \frac{\partial (F,G)(x,y)}{\partial (x,y)} \cdot C(x,y) \ -
$$
$$ - \bar{\mu}(\bar{x},\bar{y}) \cdot \frac{\partial C(x,y)}{\partial (x,y)} \Big) \cdot \Big(A(x,y) - \bar{\mu}(\bar{x},\bar{y}) \cdot C(x,y) \Big)^{-1} \cdot (\bar{\mu}(\bar{x}, \bar{y}) \cdot D(x,y) - B(x,y)) $$
where $N$ is a metric space of $1 \times 12$ vector functions $\bar{\eta}$, satisfying
\newline
1) $\eta(x,y)$ is a continuous vector function in $W$ 
\newline
2) $\sup _{\{x \in S^1 \times [0,1] \times S^1, \ y>0 \}} || y \cdot \eta(x,y) ||  \leq J$ for some $J>0$
\newline
3) $\sup _{\{x \in S^1 \times [0,1] \times S^1 \}} || y \cdot \eta(x,y) || =0 $ as $y \rightarrow 0$
\newline
For fixed $\bar{\eta}$ the mapping $P^{\bar{\mu}}$ is pointwise continuous, i.e. $\bar{\mu}_n \rightarrow \bar{\mu}$ implies $P^{\bar{\mu}_n}(\bar{\eta}) \rightarrow P^{\bar{\mu}}(\bar{\eta})$, and also a contraction on $N$, since 
$$||P^{\bar{\mu}}(\bar{\eta}_1) -  P^{\bar{\mu}}(\bar{\eta}_2) || = $$
$$\Big|\Big| (A-\bar{\mu}C)^{-1} \cdot \Big((\bar{\eta}_1 - \bar{\eta}_2) \cdot  \frac{\partial (F,G)(x,y)}{\partial (x,y)} \cdot D \Big) - (A-\bar{\mu}C)^{-1} \cdot \Big((\bar{\eta}_1 - \bar{\eta}_2) \cdot  \frac{\partial (F,G)(x,y)}{\partial (x,y)} \cdot C \Big) \cdot (A-\bar{\mu}C)^{-1} \cdot (\mu(\bar{x}, \bar{y}) \cdot D(x,y) - B(x,y)) \Big|\Big| \leq  $$
$$||\bar{\eta}_1 - \bar{\eta}_2|| \cdot \Bigg(||(A-\bar{\mu}C)^{-1}|| \cdot || \frac{\partial (F,G)(x,y)}{\partial (x,y)} || \cdot ||D|| + ||(A-\bar{\mu}C)^{-1}||^2 \cdot || \frac{\partial (F,G)(x,y)}{\partial (x,y)} || \cdot ||C|| \cdot (||\bar{\mu}|| \cdot ||D|| + ||B||) \Bigg) \leq $$
$$||\bar{\eta}_1 - \bar{\eta}_2|| \cdot \Bigg((A_3^{-1} - M \cdot C_2 \cdot |y|^{-\frac{\varsigma(b_0(r,\phi))}{p(b_0(r,\phi))}} \cdot |\ln y|)^{-1}  \cdot || \frac{\partial (F,G)(x,y)}{\partial (x,y)} || \cdot ||D|| + 
$$
$$+(A_3^{-1} - M \cdot C_2 \cdot |y|^{-\frac{\varsigma(b_0(r,\phi))}{p(b_0(r,\phi))}} \cdot |\ln y|)^{-2}  \cdot || \frac{\partial (F,G)(x,y)}{\partial (x,y)} || \cdot ||C|| \cdot (||\bar{\mu}|| \cdot ||D|| + ||B||) \Bigg) $$
Hence, the Lipschitz constant is smaller than $1$ for $y$ small enough and moreover is independent of $\bar{\mu}$.
\newline
In conclusion, $P^{\bar{\mu}}$ has a unique fixed point for every $\mu$ and this proves as in \cite{Shashkov} that $\mu_0$ is a $C^1$ smooth field of hyperplanes. 
This is due to the following lemmas (Lemma 5, 6, 7 from \cite{Shashkov}):
\begin{lemma}
The space $N$ can be equipped with a norm equivalent to the original norm and such that for all $\mu \in V$, operator $P^{\mu}$ are contraction operators with the same contraction factor $q$.
\end{lemma}

\begin{lemma}
Let $W_1$ and $W_2$ be the metric spaces, and suppose that $W_2$ is complete. Let a map $Q: W_1\rightarrow W_1$ have a unique fixed point $v^*$ to which any sequence of the form $v^{n+1}=Q \cdot v^n$ is convergent. In addition, let a contraction map $\Omega^v: W_2 \rightarrow W_2$ be associated with any element $v \in W_1$, and let the family of maps $\Omega^v$ have the following properties:
\newline
1) there exists a common contraction factor for all $\Omega^v$
\newline
2) the family of maps $\Omega^v$ depends on $v$ continuously, i.e. if $v^n \rightarrow v^*$ as $n\rightarrow +\infty$, then $\Omega^{v^n}[w] \rightarrow \Omega^{v^*}$ as $n\rightarrow +\infty$ for any $w \in W_2$
\newline
Then the map $R:W_1 \times W_2 \rightarrow W_1 \times W_2$ given by the formula $R(v,w):=(Qv,Q^v[w])$ has a unique fixed point $(v^*,w^*)$ and the sequence $(v^{n+1},w^{n+1})= R(v^{n},w^{n})$ is convergent to $(v^*,w^*)$, for an arbitrary initial condition $(v^0,w^0)$.
\end{lemma}
\begin{lemma}
The field $\mu_0$ is a smooth field o hyperplanes tangent to a $C^1$ foliation $\mathcal{L}_0$ with $C^2$ leaves.
\end{lemma}
Moreover, as shown in \cite{Shashkov}, $\mu_0$ corresponds to the $\varpi$ - invariant foliation $\mathcal{L}$ of the $(v_1,z,r,\phi)$ space (lemma 2 in \cite{Shashkov}). 
\newline
The only missing part in our reasoning is to show that in the limit $y=v_2 \rightarrow 0$ different leaves of the foliation correspond to different points $(v_2=0,z,r,\phi)$. In order to prove this, we need to show that the flow governed by the equation
\begin{equation}\label{}
\frac{d x}{d y} =  \mu(x,y), \ \ \ \ \ \text{ with } \ \ y=v_2
\end{equation}
can be continuously extended to the manifold $\{ y=0 \}$. The vector field $\mu(x,y)$ is $C^1$ for $y \neq 0$ and moreover 
\newline
1) $\mu_0(x,y)$ is a continuous vector function in $W$ 
\newline
2) $\sup _{\{x, \ y>0 \}} || \frac{\mu_0(x,y)}{\ln y} ||  \leq E$ for some $E>0$
\newline
3) $\sup _{x} || \frac{\mu_0(x,y)}{\ln y} || =0 $ as $y \rightarrow 0$
\\ \\
Let us perform the substitution 
$$ s := - (\ln y)^{-1}$$
then 
$$\frac{d s}{d y} = y^{-1} \cdot (\ln y)^{-2}$$
which gives 
\begin{equation}\label{}
\frac{d x}{d s} =  \mu(x, y) \cdot y \cdot (\ln y)^2 = - \mu(x,e^{-\frac{1}{s}}) \cdot e^{-\frac{1}{s}} s^{-2}
\end{equation}
now, due to $s \rightarrow 0 $ as $y \rightarrow 0 $ and 
$$||\mu(x,y) \cdot y \cdot (\ln y)^2|| = O(y \cdot (\ln y)^3)$$ 
we can extend smoothly the vector field 
$$\tilde{\mu}(x,s):= \mu(x,e^{-\frac{1}{s}}) \cdot e^{-\frac{1}{s}} s^{-2}$$ 
onto the manifold $\{ y=0 \} = \{ s=0 \}$. For this it suffices to see that the following definitions provide the required $C^1$ smooth extension of $\tilde{\mu}$:
$$\tilde{\mu}(x,0) := \lim _{s \rightarrow 0} \tilde{\mu}(x,s) = 0$$
$$\frac{\partial \tilde{\mu}}{\partial s} (x,0) : = \lim _{s \rightarrow 0} \frac{\partial \tilde{\mu}}{\partial s} (x,s) = 0$$
since $\frac{\partial \mu}{\partial y} = O(y^{-1})$ and 
$$\frac{\partial \tilde{\mu}}{\partial s} (x,s) = \frac{\partial \tilde{\mu}}{\partial y} \cdot \frac{\partial y}{\partial s} = y \cdot (\ln y)^2 (\frac{\partial \mu}{\partial y} \cdot y \cdot \ln y + \mu \cdot \ln y \cdot (2 + \ln y)) = O(y \cdot (\ln y)^5) 
$$
Now the extension of the flow onto the manifold $\{ y=0 \} = \{ s=0 \}$ follows from Picard's Theorem, since $\tilde{\mu}(x,s)$ is locally Lipshitz (in particular for $s$ in the neighbourhood of $0$, i.e. for $s<0$ take $\tilde{\mu}(x,s) = - \tilde{\mu}(x,-s)$ ).
\end{proof}

Although the truncated return map (\ref{truncatedReturnformula}) does not have a form of a Henon-like map in the original coordinates $(z,r,\phi)$, but for suitable choice of parameters it is conjugated to the Henon-like map.
\begin{prop}\label{HenonLike}
There exist coefficients $\Omega(r,\phi)$, $\Gamma$, $b(r,\phi)$, $c(r,\phi)$, such that the mapping
\begin{equation}\label{}
\begin{aligned}
T: \ \ \left( \begin{array}{c}
z \\
r \\ 
\phi \\ 
\end{array} \right) 
\mapsto
\left( \begin{array}{c}
\Omega(r,\phi) + \Gamma \cdot z \ \ (mod \ 1) \\
b(r,\phi) \\
c(r,\phi) + z \ \ (mod \ 1) \\ 
\end{array} \right) 
\end{aligned}
\end{equation}
has a fixed point with a normal form, in a suitable coordinates, given by:
\begin{equation}\label{}
\begin{aligned}
\left( \begin{array}{c}
x \\
y \\ 
w \\ 
\end{array} \right) 
\mapsto
\left( \begin{array}{c}
y \\
w \\
x + y - w + A \cdot y^2 + B \cdot y \cdot w + C \cdot w^2 + h.o.t.
\end{array} \right) 
\end{aligned}
\end{equation}
with $(C - A) \cdot (A - B + C) > 0$.
Hence, by \cite{LorenzLike1}, \cite{LorenzLike2} and \cite{Shimizu} there exists arbitrarily small perturbation of this mapping possessing a Lorenz-like attractor.
\end{prop}
\begin{proof}
Let us choose $\Omega(r,\phi)$, $\Gamma(r,\phi)$, $b(r,\phi)$, $c(r,\phi)$ of the special form:
$$\Omega(r,\phi) = a_1 r + a_2 \phi \cdot (1-\phi) + a_3 r^2 
$$
$$
\Gamma = constant \in \mathbb N$$
$$b(r,\phi) = b_1 r + b_2 \phi \cdot (1-\phi) + b_3 r^2 + b_4 r \phi \cdot (1-\phi) + b_5 (\phi \cdot (1-\phi))^2$$
$$c(r,\phi) = c(r) = c_1 r$$
Under the above assumptions, the map $T$ possesses a fixed point $(0,0,0)$. Moreover, this allows us to treat $z$ variable $(mod \ 1)$. We perform the change of the coordinates as follows:
\begin{equation}\label{}
\begin{cases}
x := \phi
\\
y : = z + c(r)
\\
w := \Omega(r,\phi) + \Gamma \cdot z + c \big(b(r,\phi) \big)
\end{cases}
\end{equation}
In the coordinates $(x,y,w)$ the map $T$ reads as:
\begin{equation}\label{}
\begin{aligned}
T: \ \ \left( \begin{array}{c}
x \\
y \\ 
w \\ 
\end{array} \right) 
\mapsto
\left( \begin{array}{c}
y \\
w \\
\Omega \big(b(r,\phi), c(r) +z \big) + \Gamma \cdot z + c \big( b(b(r,\phi), c(r) +z ) \big) \\ 
\end{array} \right) 
\end{aligned}
\end{equation}
Utilizing the chain rule, we find the expansion of $\bar{w}$ in terms of $(x,y,w)$ around fixed point $(0,0,0)$:
$$\bar{w} = \Omega \big( b(r,\phi), c(r) +z \big) + \Gamma \cdot z + c \big(b(b(r,\phi), c(r) +z ) \big) = 
T_3(x,y,w)= $$
$$= T_3(0,0,0) + \big(T_3(0,0,0)\big)_{x}' \cdot x + \big(T_3(0,0,0)\big)_{y}' \cdot y + \big(T_3(0,0,0)\big)_{w}' \cdot w + \big(T_3(0,0,0)\big)_{xx}'' \cdot \frac{x^2}{2} + \big(T_3(0,0,0)\big)_{yy}'' \cdot \frac{y^2}{2}+ $$
$$+ \big(T_3(0,0,0)\big)_{ww}'' \cdot \frac{w^2}{2}  + \big(T_3(0,0,0)\big)_{xy}'' \cdot x \cdot y + \big(T_3(0,0,0)\big)_{xw}'' \cdot x \cdot w + \big(T_3(0,0,0)\big)_{yw}'' \cdot y \cdot w + h.o.t. $$
We have $T_3(0,0,0) = 0$ and moreover the following conditions are satisfied:
$$\big(T_3(0,0,0)\big)_{x}' = 1, \ \ \big(T_3(0,0,0)\big)_{y}' = 1, \ \ \big(T_3(0,0,0)\big)_{w}' = -1$$
$$\big(T_3(0,0,0)\big)_{xx}'' = 0, \ \ \big(T_3(0,0,0)\big)_{xy}'' = 0, \ \ \big(T_3(0,0,0)\big)_{xw}'' = 0$$ 
provided that 
$$a_2 = -b_1 -2 \Gamma$$
$$c_1 = \frac{(1-a_2)(1+b_1)}{b_2 (1+b_1+2 \Gamma)}$$
$$a_1 = c_1 \frac{2 \Gamma -b_1 - b_1^2}{1+b_1}$$
$$b_5 = -\Big( b_2 (a_3 b_2^2 (1 + b_1^2 (-1 + \Gamma)^2 - 4 \Gamma + 2 b_1 (-1 + \Gamma) \Gamma +  5 \Gamma^2) - (-1 + b_1) \Gamma ((-1 + a_5 + b_1 - a_5 b_1) \Gamma + 
$$
$$
b_4 (-1 + 2 \Gamma) (1 + b_1 + 2 \Gamma)) +  b_2^2 c_3 (b_1^4 (-1 + \Gamma)^2 + 2 b_1 (-1 + \Gamma) \Gamma +        2 b_1^3 (-1 + \Gamma) \Gamma + \Gamma (-2 + 9 \Gamma - 8 \Gamma^2) + 
$$
$$      
			b_1^2 (1 - 2 \Gamma + 2 \Gamma^2)) + 
    b_2 (a_4 \Gamma (-1 + b_1 + 2 \Gamma - 2 b_1 \Gamma) + 
       b_3 (1 + b_1^3 (-1 + \Gamma)^2 - 2 \Gamma - 3 \Gamma^2 + 8 \Gamma^3 + 
$$
$$
          b_1^2 (1 - 4 \Gamma + 3 \Gamma^2) + b_1 (1 - 6 \Gamma + 7 \Gamma^2)))) \Big) \Big{/} \Big( (-1 + b_1)^2 \Gamma^2 (1 + b_1 + 2 \Gamma) \Big)$$
$$b_3 = \Big(-b_1 b_4 + b1^3 b_4 + 2 b_1^3 b_2^2 c_3 + 2 b_1^4 b_2^2 c_3 + b_1 b_4 \Gamma - 
    b_1^3 b_4 \Gamma - 4 b_2^2 c_3 \Gamma - 2 b_1^2 b_2^2 c_3 \Gamma - 
    4 b_1^3 b_2^2 c_3 \Gamma - 2 b_1^4 b_2^2 c_3 \Gamma - 2 b_4 \Gamma^2 + 
$$
$$		
    2 b_1 b_4 \Gamma^2 + 8 b_2^2 c_3 \Gamma^2 + a_4 b_1 b_2 (-1 + b_1 + \Gamma - b_1 \Gamma) -
    2 a_3 b_2^2 (-1 + b_1^2 (-1 + \Gamma) + 2 \Gamma + b1 \Gamma)\Big) \Big{/} \Big(2 b_2 (-1 + 
$$
$$
      b_1^3 (-1 + \Gamma) + 4 \Gamma^2 + b_1^2 (-1 + 2 \Gamma) + b_1 (-1 + 3 \Gamma)) \Big)$$
$$b_4 = \Big(-b_2 (a_4 (-1 + b_1) (\Gamma + b_1^4 (-1 + \Gamma)^2 \Gamma + 
            b_1^5 (-1 + \Gamma)^2 \Gamma - 6 \Gamma^3 + 8 \Gamma^5 + 
            b_1^3 (-1 + 4 \Gamma - 3 \Gamma^2 - 3 \Gamma^3 + 2 \Gamma^4) + 
$$
$$
            b_1 (-1 + 5 \Gamma - 5 \Gamma^2 - 6 \Gamma^3 + 6 \Gamma^4) + 
            b_1^2 (-2 + 8 \Gamma - 4 \Gamma^2 - 15 \Gamma^3 + 12 \Gamma^4)) + 
         2 b_2 (b_1^7 c_3 (-1 + \Gamma)^2 \Gamma - 
$$
$$
            \Gamma (1 + \Gamma) (-1 + 2 \Gamma)^3 (2 a_3 + c_3 - 2 c_3 \Gamma) + 
            b_1^2 c_3 (1 - 2 \Gamma - 2 \Gamma^2 + 5 \Gamma^3) + 
            b_1^5 c_3 (-1 + 3 \Gamma - \Gamma^2 - 4 \Gamma^3 + 2 \Gamma^4) + 
$$
$$
            b_1^4 c_3 (-1 + 3 \Gamma + \Gamma^2 - 11 \Gamma^3 + 8 \Gamma^4) + 
            b_1^3 c_3 (1 - 3 \Gamma + 6 \Gamma^3 - 6 \Gamma^4 + 4 \Gamma^5) + 
            b_1 \Gamma (a_3 (2 - 6 \Gamma + 8 \Gamma^3) + 
$$
$$
               c_3 (1 - 9 \Gamma + 15 \Gamma^2 + 8 \Gamma^3 - 20 \Gamma^4))))\Big) \Big{/} \Big((-1 + 
         b_1) (1 + b_1) (\Gamma + b_1^4 (-1 + \Gamma)^2 \Gamma + 
         b_1^5 (-1 + \Gamma)^2 \Gamma - 2 \Gamma^2 + 
$$
$$
         b_1^3 (-1 + 4 \Gamma - 3 \Gamma^2 - 3 \Gamma^3 + 2 \Gamma^4) + 
         b_1 (-1 + 5 \Gamma - 7 \Gamma^2 - 4 \Gamma^3 + 10 \Gamma^4) + 
         b_1^2 (-2 + 8 \Gamma - 4 \Gamma^2 - 15 \Gamma^3 + 12 \Gamma^4)) \Big)$$			
If we choose for example:
$$a_3 = 1, \ \ b_1 = -2, \ \ b_2 = -1, \ \ \Gamma = 5$$
then
$$\big(T_3(0,0,0) \big)_{yw}'' = B,  \ \ \big(T_3(0,0,0) \big)_{yy}'' = 2 \cdot A, \ \ \big(T_3(0,0,0) \big)_{ww}'' = 2 \cdot C $$
satisfy
$$(C-A) \cdot (A-B+C) >0$$
\end{proof}

\subsection{Further analysis of the return map}\label{partialHyp}
In the latter proposition we assumed that $\Gamma \in \mathbb N$ in order to compute normal form of the map $T$ and to treat $z$ variable $(mod \ 1)$. The proposition \ref{coneField} below, provides conditions under which the truncated return map (\ref{truncatedReturnformula}) is partially hyperbolic.
\begin{prop}\label{coneField}
For the coefficients $\Omega(r,\phi)$, $\Gamma(r,\phi)$, $b(r,\phi)$, $c(r,\phi)$, such that $\Omega_r '(r,\phi)$, $\Omega_{\phi} '(r,\phi)$, $\Gamma_r '(r,\phi)$, $\Gamma_{\phi} ' (r,\phi)$, $b_r '(r,\phi)$, $b_{\phi} '(r,\phi)$, $c_r ' (r,\phi)$, $c_{\phi} ' (r,\phi)$ are small enough (uniformly in $(r,\phi)$), the mapping $T$ possesses an invariant cone-field 
$$\mathcal{C}_{L, (z,r,\phi)}:= \Big{\{} v= v_1 + v_{23} \in T_{(z,r,\phi)}M \ | \ v_1=(v_{11},0,0), \ v_{23}=(0,v_2,v_3) \ \text{such that } ||v_{23}|| < L\cdot |v_1| \Big{\}}$$ 
where $M := \mathbb R \times [0,1] \times S^1$, $L<1$ and $|| \cdot ||$ denotes the Euclidean norm. 
\end{prop}
\begin{proof}
Let $\Gamma(r,\phi) = \Gamma \in \mathbb N$, $\Gamma >1$. Hence at point $(z,r,\phi) \in M$ the derivative of $T$ is given by 
\begin{equation}\label{}
\begin{aligned}
D_{(z,r,\phi)} T = \left( \begin{array}{ccc}
\Gamma & \Omega_r ' (r,\phi) & \Omega_{\phi} ' (r,\phi) \\
0 & b_r ' (r,\phi) & b_{\phi} ' (r,\phi) \\
1 & c_r ' (r,\phi) & c_{\phi} ' (r,\phi) \\
\end{array} \right)
\end{aligned}
\end{equation}
and so, for the $v=(v_1,v_2,v_3) \in \mathcal{C}_{L,(z,r,\phi)}$, we have 
\newline
$$D_{(z,r,\phi)} T \cdot v= ( \ \Gamma \cdot v_1 + \Omega_r ' (r,\phi) \cdot v_2 + \Omega_{\phi} ' (r,\phi) \cdot v_3, \  b_r ' (r,\phi) \cdot v_2 + b_{\phi} ' (r,\phi) \cdot v_3, \ v_1 + c_r ' (r,\phi) \cdot v_2 + c_{\phi} ' (r,\phi) \cdot v_3 \ )$$
Now the invariance condition for the cone-field $\mathcal{C}_{L, (z,r,\phi)}$ means that:
$$||( \ b_r ' (r,\phi) \cdot v_2 + b_{\phi} ' (r,\phi) \cdot v_3, \ v_1 + c_r ' (r,\phi) \cdot v_2 + c_{\phi} ' (r,\phi) \cdot v_3 \ ) || < L \cdot | \ \Gamma \cdot v_1 + \Omega_r ' (r,\phi) \cdot v_2 + \Omega_{\phi} ' (r,\phi) \cdot v_3 \ |$$
From the Cauchy-Schwarz inequality we have:
$$|\ \Omega_r ' (r,\phi) \cdot v_2 + \Omega_{\phi} ' (r,\phi) \cdot v_3 \ | \leq ||(v_2,v_3)|| \cdot \sqrt{(\Omega_r ' (r,\phi))^2 + (\Omega_{\phi} ' (r,\phi))^2} < L \cdot |v_1| \cdot \sqrt{(\Omega_r ' (r,\phi))^2 + (\Omega_{\phi} ' (r,\phi))^2} $$
choose 
$$c_1 \in \big(0, 1 \big), \ \ \ L \in \Bigg(0, \frac{c_1 \cdot \Gamma}{\sqrt{(\Omega_r ' (r,\phi))^2 + (\Omega_{\phi} ' (r,\phi))^2}} \ \Bigg)$$
then
$$|\ \Omega_r ' (r,\phi) \cdot v_2 + \Omega_{\phi} ' (r,\phi) \cdot v_3 \ | < L \cdot |v_1| \cdot \sqrt{(\Omega_r ' (r,\phi))^2 + (\Omega_{\phi} ' (r,\phi))^2} < c_1 \cdot \Gamma \cdot |v_1| $$
which leads to
$$(1-c_1) \cdot \Gamma \cdot |v_1| < | \ \Gamma \cdot v_1 + \Omega_r ' (r,\phi) \cdot v_2 + \Omega_{\phi} ' (r,\phi) \cdot v_3 \ | $$
On the other hand:
\newline
$$\Big|\Big|\Big( \ b_r ' (r,\phi) \cdot v_2 + b_{\phi} ' (r,\phi) \cdot v_3, \ v_1 + c_r ' (r,\phi) \cdot v_2 + c_{\phi} ' (r,\phi) \cdot v_3 \ \Big) \Big|\Big|^2 = $$
$$ = v_1^2 + \Big(\ b_r ' (r,\phi) \cdot v_2 + b_{\phi} ' (r,\phi) \cdot v_3 \Big)^2 + \Big( c_r ' (r,\phi) \cdot v_2 + c_{\phi} ' (r,\phi) \cdot v_3 \Big)^2 + 2\cdot v_1 \cdot \Big(c_r ' (r,\phi) \cdot v_2 + c_{\phi} ' (r,\phi) \cdot v_3 \Big) $$
$$ \leq v_1^2 + \Big((\ b_r ' (r,\phi))^2 + (b_{\phi} ' (r,\phi))^2 + (\ c_r ' (r,\phi))^2 + (c_{\phi} ' (r,\phi))^2 \Big) \cdot (v_2^2 + v_3^2) + 2\cdot v_1 \cdot  \sqrt{(c_r ' (r,\phi))^2 + (c_{\phi} ' (r,\phi))^2} \cdot \sqrt{v_2^2 + v_3^2}  $$
$$ < v_1^2 \cdot \Bigg( 1+\ \Big((\ b_r ' (r,\phi))^2 + (b_{\phi} ' (r,\phi))^2 + (\ c_r ' (r,\phi))^2 + (c_{\phi} ' (r,\phi))^2 \Big) \cdot L^2 + 2 \cdot L \cdot \sqrt{(c_r ' (r,\phi))^2 + (c_{\phi} ' (r,\phi))^2} \Bigg)$$
\newline
we want the right hand side of the last inequality to be smaller than 
$$L^2 \cdot (1-c_1)^2 \cdot \Gamma^2 \cdot v_1^2$$
which in consequence means:
$$1+\ \Big((\ b_r ' (r,\phi))^2 + (b_{\phi} ' (r,\phi))^2 + (\ c_r ' (r,\phi))^2 + (c_{\phi} ' (r,\phi))^2 \Big) \cdot L^2 + 2 \cdot L \cdot \sqrt{(c_r ' (r,\phi))^2 + (c_{\phi} ' (r,\phi))^2} <$$
$$<  L^2 \cdot (1-c_1)^2 \cdot \Gamma^2$$
Since $\Gamma>1$, the last inequality has a positive solution in 
$$c_1 \in \big(0, 1 \big), \ \ \ L \in \Bigg(0, \frac{c_1 \cdot \Gamma}{\sqrt{(\Omega_r ' (r,\phi))^2 + (\Omega_{\phi} ' (r,\phi))^2}} \ \Bigg)$$
for
$$\Omega_r ' (r,\phi), \ \Omega_{\phi} ' (r,\phi) , \  b_r ' (r,\phi), \ b_{\phi} ' (r,\phi), \ c_r ' (r,\phi), \ c_{\phi} ' (r,\phi)$$ 
being sufficiently small (uniformly in $(r,\phi)$).
\end{proof}

\begin{cor}
If $T$ - the truncated return map (\ref{truncatedReturnformula}) possesses an invariant cone-field, as in proposition \ref{coneField}, then there exists a $T$ - invariant foliation $\mathcal{L}= \{ h(z) \}$ of the phase space $(z,r,\phi)$ with the leaves given by the graphs of the functions $(r,\phi)=h(z)$ 
\end{cor}

\section{Acknowledgements}
C.Olszowiec is grateful to Maciej Capi\'nski, Josep-Maria Mondelo and Sina T{\"u}reli for the discussions and to Polish Academy of Sciences, Centre de Recerca Matem\`atica and Imperial College Roth Studentship scheme for financial support.



\end{document}